\documentclass[11pt,english]{article}

\usepackage[all]{xy}
\usepackage[T1]{fontenc}
\usepackage[latin1]{inputenc}
\usepackage{babel}

\usepackage{setspace}
\usepackage{indentfirst}

\makeatletter

\usepackage{verbatim}

\makeatother

\usepackage{enumerate}
\usepackage{amsmath,amscd}
\usepackage{amssymb}
\usepackage[all]{xy}

\usepackage{amsmath,amscd,amsthm}

           {\nolinebreak $\Box$ \end{trivlist}}

\newtheorem{prop}{Proposition}[section]
\newtheorem{lem}[prop]{Lemma}

\newtheorem{cor}[prop]{Corollary}
\newtheorem{them}[prop]{Theorem}

\newtheorem{defi}[prop]{Definition}

\newtheorem{nota}[prop]{Notation}

\theoremstyle{definition}
\newtheorem{rema}[prop]{Remark}
\newtheorem{exampl}[prop]{Example}
\newenvironment{rem}{ \begin{rema}   }{$\qed $ \end{rema}}
\newenvironment{examp}{ \begin{exampl}   }{$\qed $ \end{exampl}}

\newcommand{\Cechs}{ N[{\mathcal U}] }

\newcommand{\toto}{\rightrightarrows}
\newcommand{\BBB}{B}

\newcommand{\id}{{\mathrm{ Id}}}
\newcommand{\Aut}{{\mathrm{ Aut}}}

\newcommand{\G}{{\mathbf G}}
\newcommand{\gggg}{{\mathbf g}}

\newcommand{\vvvv}{{\mathbf v}}
\newcommand{\fatg}{{\mathbf g}}

\newcommand{\e}{{\mathrm{ e}}}

\newcommand{\GGG}{{\mathcal G}}

\newcommand{\lra}{{\longrightarrow}}
\newcommand{\HH}{{\mathbf H}}

\newcommand{\gpdproductR}{\, \bullet_{\mathcal R} \,}
\newcommand{\gpdactkernel}{\, \bullet_{{\mathcal R},{K}} \, }
\newcommand{\gpdactband}{\, \bullet_{{\mathcal R},{Band}} \,}
\newcommand{\gpdactP}{\, \bullet_{{\mathcal R},{P}} \,}

\begin{document}

\title{Lie groupoids and crossed module-valued gerbes over stacks}
\author{M.Jawad Azimi}

\date{\today}
\date{\texttt{}}
\maketitle
CMUC, Department of Mathematics,\\ University of Coimbra, 3001-454
Coimbra, Portugal\\

\begin{abstract}
We give a precise and general description of  gerbes valued in  arbitrary crossed module and over an arbitrary differential stack.
We do it using only Lie groupoids, hence ordinary differential geometry. We prove the coincidence with the existing notions by comparing our
 construction with non-Abelian cohomology.
\end{abstract}

\tableofcontents
\section{Introduction}
Differential gerbes appeared from the very beginning  as being classes in some "higher" cohomology \cite{Giraud}. For instance, non-Abelian gerbes correspond to non-Abelian $1$-cohomology in the sense of Dedecker \cite{BaezSchreiber,BM,Dedecker}. This is also the form under which it appears in theoretical physics \cite{GR,Witten}.
 But differential gerbes can also be thought of as being a certain class of bundles over a differential stack, and, to quote \cite{BehrendXu}, "there is a dictionnary between differential stacks and Lie groupoids".
The purpose of the present article is to add one entry to that dictionary, namely to define with great care, in terms of Lie groupoids and for all
crossed module $\G \to \HH $, the notion of $\G \to\HH $-gerbes  and to justify that definition by showing the coincidence of the notion introduced with non-Abelian $1$-cohomology.

There are, of course, several other manners to define non-Abelian gerbes, and to state their properties.
In a recent work \cite{Waldorf} these numerous definitions have been carefully enumerated and shown, in a rigorous manner, to coincide. More precisely, the authors of \cite{Waldorf} have merged four definitions of smooth $\Gamma$-gerbes, with $\Gamma$ a strict $2$-groups (notice that strict $2$-groups are indeed in one-to-one correspondence with crossed module):
\begin{enumerate}
\item smooth $\Gamma$-valued $1$-cocycles (for which they refer to \cite{BM}, but which matches by construction the definition in terms of non-Abelian cohomology in the sense of \cite{Dedecker} just mentioned), see also \cite{AJC}.
\item classifying maps valued in the realization $B\Gamma$ of the simplicial tower of $\Gamma $,
\item bundle gerbes in the sense of \cite{AschieriCantiniJurco},
\item principal $\Gamma$-bundles in the sense of Bartels \cite{Bartels}, the idea being to generalize the notion of principal bundle from Lie groups to Lie (strict) $2$-groups. This point of view was also used in \cite{GinotStienon}, and we will relate our construction to their
construction in due time.
\end{enumerate}
 But, as mentioned in \cite{Waldorf}, example 3.8, there is in the particular case of $\G \to Aut(\G)$-gerbes over manifolds a fifth equivalent definition which is in terms of Lie groupoid extensions.
  We can restate our purpose by saying that it consists in giving this fifth description in the general setting of arbitrary crossed module (and not only $\G \to \Aut (\G)$). Also, we give a definition that makes sense  when the base space is not a manifold but an arbitrary differential stack, reaching therefore
the same level of generality as \cite{GinotStienon} (we simply claim to be slightly more precise about the problem of
identifications of a priori different extensions defining the same gerbe).

  Although the four approaches just mentioned can be remarkably effective in the sense that the objects have short, simple and workable definitions, it always requires a deep familiarity with category theory (or even toposes and higher categories) making them hardly accessible for a mathematician not used  to these techniques. Our manner is maybe more difficult in the sense that the objects are always defined as classes of -oids up to Morita equivalences, which sometimes yield long definition, and forces us to check that  properties are Morita invariant, but it is certainly simpler in the sense that it uses
the ordinary language of differential geometry (manifolds, -oids, maybe \v{C}ech cocycles) from the beginning to the end. We can not claim that we avoid all categorical language, since groupoids are categories, but we use comparatively much less involved categorical tools.

  The present work is also in the continuation of \cite{BehrendXu} (where $S^1 $-gerbes over a differential stack are extensively studied using this Lie groupoid point of view), of \cite{LaurentStienonXu} (where the case of non-Abelian$\G\to\Aut(\G)$-gerbes over Lie groupoids is investigated, but the correspondence with non-Abelian  $1$-cocycles is not dealt with very precisely,
 and of \cite{BL}, (where the previous construction is investigated in detail for $\G$-gerbes and extended to connections). Our work is definitively in the same line of those, but there are important differences that we now outline. Abelian gerbes in the sense of \cite{BehrendXu} (resp. $\G$-gerbes in the sense of \cite{LaurentStienonXu}) corresponds to the case where the crossed modules in which the gerbe takes values is $S^1 \to pt$ (resp. $\G \to \Aut (\G)$), so that our work generalizes both. Second, we made more precise in the present article the notion of gerbes over an object (manifold, Lie groupoid, or differential stack). This means that, unlike \cite{LaurentStienonXu}, we do not simply define gerbes as being $\G$-extensions up to Morita equivalence, and this is for two reasons:
\begin{enumerate}
\item first, as already stated, we wish to make precise over what object our gerbe is, which
means that we only allow ourself to take Lie groupoid extensions $X \to Y $ where
the "small" Lie groupoid $Y$ is itself "over" a given object $B$ (manifold or Lie groupoid or differential stack).
 By "over", we mean that "Y" is obtained by taking a pull-back of $B$. Also, Morita equivalence should be taken in such a way that the base manifold or groupoid is not "changed". This last issue is easily understandable, and always appear in differential geometry: the space of principal bundles over a manifold $M$, in a similar fashion, is not obtained by considering all possible principal bundles $P \to M$ modulo principle bundles isomorphisms, but modulo principal bundles isomorphisms over the identity of $M$.
\item second, when taking an arbitrary crossed module $\G \to \HH $, Lie groupoid extensions are not enough.
By spelling out the manifold case, and knowing that we wish to have a correspondence with crossed module valued non-Abelian $1$-cocyle, we arrived at the conclusion that we need to consider a Lie groupoid $\G$-extension together with a principal $\HH$-bundle. These two structures are not independent, and, having in mind the manifold case again, one sees that we need this principle bundle to be equipped with a principal bundle morphism taking values in the band of the Lie groupoid extension, map on which still two constraints have to be imposed.
\end{enumerate}

The paper is organized as follows.
In section \ref{sec:GpdExt}, we recall from \cite{LaurentStienonXu} the notion of $\G$-extensions of Lie groupoids, i.e. a surjective submersion morphism
 of Lie groupoids over the same base ${\mathcal R} \stackrel{\phi}{\to} \GGG $, for which the kernel is a locally trivial bundle of groups with typical fiber $ \G$. We then recall, following \cite{LaurentStienonXu}, the notion of the band of the $\G$-extension, which is some principal bundle over the Lie groupoid ${\mathcal R}$. We then define $ \G \to \HH$-extensions, namely $\G$-extensions ${\mathcal R} \stackrel{\phi}{\to} \GGG $ endowed with some
principal $ \HH$-bundle which admits the band as a quotient, see definition \ref{def: Azimi-extension} for a more precise description.

We then recall the definition of Dedecker's non-Abelian $1$-cocyle (resp. non-Abelian $1$-coboundaries, non-Abelian $1$-cohomology)
on an open cover of a given manifold $N$ and describe a dictionary between these objects and $ \G \to \HH$-extensions.
More precisely, we define, given an open cover of a manifold, a subclass of $ \G \to \HH$-extensions called
adapted $ \G \to \HH$-extensions of the \v{C}ech groupoid, and we show the following points, given an open cover ${\mathcal U} $ on the manifold $ N$:
\begin{itemize}
 \item \emph{Proposition \ref{prop:cocycles}}, There is a one-to-one correspondence between:
         \begin{enumerate}
          \item[(i)] $ \G \to \HH$-valued non-Abelian $1$-cocycles w.r.t. ${\mathcal U} $
          \item[(ii)] adapted $ \G \to \HH$-extensions of the \v{C}ech groupoid $ N[{\mathcal U}]$
         \end{enumerate}
 \item \emph{Proposition \ref{prop:coboundaries}}, There is a one-to-one correspondence between:

         \begin{enumerate}
          \item[(i)] $ \G \to \HH$-valued non-Abelian $1$-coboundaries w.r.t. ${\mathcal U} $
          \item[(ii)] isomorphisms of adapted $ \G \to \HH$-extensions of the \v{C}ech groupoid $ N[{\mathcal U}]$
         \end{enumerate}
 \item \emph{Theorem \ref{cor:1-1Cohomology-adapted}}, There is a one-to-one correspondence between:
    \begin{enumerate}
          \item[(i)] $ \G \to \HH$-valued non-Abelian $1$-cohomology w.r.t. ${\mathcal U} $
          \item[(ii)] isomorphism classes of adapted $ \G \to \HH$-extensions of the \v{C}ech groupoid $ N[{\mathcal U}]$ up to Morita
 equivalence over the identity.
\item[(iii)] (assuming the covering to be a good one) isomorphism classes of $ \G \to \HH$-extensions of the \v{C}ech
groupoid $ N[{\mathcal U}]$ up to isomorphisms over the identity of $N[{\mathcal U}]$.
         \end{enumerate}
\end{itemize}

The first purpose of section \ref{sec:MoritaExt} is to show that our constructions are independent from the choice of an open cover and to reach therefore $ \G \to \HH$-valued non-Abelian $1$-cohomology
in its full generality. This requires to define the notion of Morita equivalence of  $ \G \to \HH$-extensions, which, in turn, allows to complete the previous isomorphisms to eventually obtain the one we are really interested in:
\begin{itemize}
\item \emph{Theorem \ref{th:coho=gerbes}}, There is a one-to-one correspondence between:
    \begin{enumerate}
          \item[(i)] $ \G \to \HH$-valued non-Abelian $1$-cohomology,
\item[(ii)]    $ \G \to \HH$-extensions of a pull-back of the
groupoid $ N \toto N $ up to Morita equivalence over the identity of $N$.
         \end{enumerate}
\end{itemize}
The point of this last theorem gives a clear hint of what a $ \G \to \HH$-gerbe over a given Lie-groupoid $\BBB$ should be, namely the
$ \G \to \HH$-extensions of a pull-back of the groupoid $ \BBB $ up to Morita equivalence over the identity of $\BBB$.
The last theorem of the present article says that Morita equivalent Lie groupoids $\BBB$ and $\BBB'$ have the same
 $ \G \to \HH$-gerbe over them, making sense therefore of the notion of  $ \G \to \HH$-gerbes over a differential stack.

\subsection{Pre-requisites}
\label{sec:prerequisites}

\noindent
A \emph{crossed module of Lie groups} (consult, for instance, \cite{BaezSchreiber}) is a quadruple $(\G, \HH, \rho,\jmath)$, where $\rho:\G \rightarrow \HH $
  and $\jmath : \HH \to \Aut (\G)$ are Lie group homomorphisms satisfying the next conditions, for all $ \fatg,\fatg' \in \G, h \in \HH$
\begin{enumerate}
\item $\rho \big(  h(\fatg) \big)=h \rho(\fatg) h^{-1}$
\item $ \rho (\fatg) \big(\fatg'\big)  =\fatg \fatg' \fatg^{-1} $
\end{enumerate}
with the understanding that $h(g) $, for every $ h\in \HH, g\in \G$ is a shorthand for $\jmath (h) (g) $.
Notice that here we consider that the action of $\HH$ on $\G$ to be a left-action, which is not the usual convention, but is necessary to recover the formulas of the $\G \to \Aut (\G)$ case as they are stated in \cite{LaurentStienonXu,BL}.
 In order to avoid an easily done confusion between elements in $\G$ and in $ \HH$, we shall denote by bold letters,
 $\fatg,\fatg'$ elements of $\G$, and in ordinary letters $h,h'$ elements in $\HH $. Also, bold letters shall be used for
  $\G$-valued functions. Last, it is customary to denote a cross-module by $\G\stackrel{\rho}{\to}\HH$, forgetting to make explicit the morphism~$\jmath $.

\bigskip
\emph{Notations related to open covers on manifolds.}\label{notation:Cech}
For ${\mathcal U}= (U_i)_{i \in I} $ an open cover on a manifold $N$,
we use the shorthand $U_{ij} = U_i \cap U_j$ for all $i,j \in I$, and introduce the convenient notation $$U_{i_1 \dots i_n} := U_{i_1} \cap \dots \cap  U_{i_n} $$  for all $n \in {\mathbb N}$ and all $i_1,\dots,i_n \in I$. We warn the reader that $U_{ij}$ is not equal to $U_{ji}$ for $i \neq j$, and, more generally, $U_{i_1 \dots i_n} $ is not equal $U_{i_{\sigma(1)} \dots i_{\sigma(n)}} $ (for $\sigma \in \Sigma_n$ a permutation, and $i_1,\dots,i_n$ distinct).

An extremely common notation in the literature dealing with gerbes is to denote by $x_i$ (resp. $x_{ij},x_{ijk}$) an element $x\in M$ that happens to belong to some open subset $ U_i$ (resp. $U_{ij}, U_{ijk} $), when it is seen as an element in $U_i$ (resp. $U_{ij},U_{ijk} $).
  We extend this convention for all kind of objects: for instance, for a function $ \lambda$ whose domain of definition is  $\coprod_{i_1,\dots,i_n \in I}  U_{i_1 \dots i_n}$,
we write $ \lambda_{i_1\dots i_n} $ for its restriction to~$U_{i_1 \dots i_n}$.
%





\bigskip
\emph{Lie groupoids : notations and basic facts.}
  Given $ M,N, P$ smooth manifolds and $f:M \to P $, $g: N \to P$ smooth maps, we define the \emph{fibered product} to be the closed subset of $M \times N  $ made of all pairs $(m,n)$ with $f(m)=g(n)$, and we denote it by $M \times _{f,P,g} N$ in general, and sometimes by $M \times _{P} N$ when there is no risk of confusion. The following is extremely classical :
 \begin{lem} \label{lem:manifold}\cite{BergerGostiaux}
  Let $ M,N, P$ be smooth manifolds. If at least one of the smooth maps $f:M \to P $ or $g: N \to P$ is a surjective submersion, then the set $M \times _{f,P,g} N$ is a smooth manifold.
  \end{lem}
We refer to \cite{McK} for the definition of Lie groupoids, but we wish to clarify some notations. When introducing a Lie groupoid, we shall in general simply mention the names of the manifolds of objects and the manifolds of arrows, using the notation $\Gamma \toto M $. Indeed, the source, target and unit maps for all Lie groupoids $\Gamma \toto M$ shall be denoted by the same letters $s,t$ and $\epsilon$ respectively. In general, the product shall be either denoted by a fat dot $\bullet$ or simply skipped, and the exponent $-1$ shall be used for the inverse map. However, at some point, we shall have to consider pairs of manifolds that admit several different Lie groupoid structures, that, fortunately, have the same source, target and unit maps. We will then introduce a  notation for the product (and inverse) that will distinguish them.
Last, our convention is that the product $x\bullet x'$ of two elements $x,x'$ in a Lie groupoid is defined when $t(x)=s(x')$.

 A left-action of Lie groupoid $\mathcal B \toto \mathcal B_0$ on a  manifold $X$ with respect to a surjective submersion $p:X \to \mathcal B_0$ is a map
      $$ \mathcal B \times_{t,\mathcal B_0,p} X\longrightarrow X,$$
 (denoted by $(b,x)\mapsto b\cdot x$) such that $p(b\cdot x)= s(b)$ and subject to the following axioms, analogous to those of group actions:
 $$b\cdot (a\cdot x)=(b  a)\cdot x \hbox{ and } \epsilon({p(x)}) \cdot x=x, $$
 for all admissible $a,b \in \mathcal B$ and $x\in X$. We shall often say action for left-action for the sake of simplicity.
   Since we may have to deal with situations where there are more than one Lie groupoid or more than one manifold involved, it will be convenient to write an action by $b\bullet_{\mathcal B, X}x$, mentioning therefore in the notation itself which groupoid acts and which manifold is acted upon.

To ensure a self-contained exposition, we recall the definition of the pull-back of a Lie groupoid.
Notice first that, for a given manifold $B$, Lie groupoids over $B$ form a category, with morphisms being
Lie groupoid morphisms over the identity of $B$. Similarly, topological groupoids over a given manifold form a category.

\begin{defi}\label{def:pback}\cite{McK}
Let $p:M \to B  $ be a smooth map. The assignments below define a functor from the category of Lie groupoids
over $B$ to the category of topological groupoids over $M$:
\begin{enumerate}
\item {\emph (On objects)}
Let $\GGG  \toto B$ be a Lie groupoid over a manifold $B$. Then the set $ \GGG[p]:= M\times_{p,B,s}\GGG\times_{t,B,p} M $ is endowed with a topological groupoid structure over $M$ given as follows: the source and target $ s, t:\GGG[p] \rightarrow M$ are the projections on the first and the third components respectively, the unit map is given for all $x \in M $ by $x \mapsto (x,\varepsilon\circ p(x),x)$, where $\varepsilon$ is the unit map of the Lie groupoid $\GGG\toto B$. Last, the multiplication and the inverse are given by:
     $$(x,\gamma, y)\bullet(y,\gamma' , z)=(x, \gamma\bullet\gamma',z) \hbox{ and }
   (x,\gamma ,y)^{-1}=(y,\gamma^{-1},x). $$
   for all $x,y,z \in M$ and $\gamma , \gamma' \in \GGG.$
\item {\emph (On arrows)} Given $\phi : \GGG \to \GGG' $ be a Lie groupoid homomorphism over the identity of $B$, we set $\phi[p,M]$ to be $(n,r,n') \mapsto (n,\phi(r),n')$ for all $(n,r,n' ) \in \GGG[p,M] = M \times_B \GGG \times_{B} M $. By construction, $\phi[p,M] $ is a  Lie groupoid homomorphism over the identity of $M$ from $ \GGG[p]$ to $ \GGG'[p] $.
 \end{enumerate}
%
     The topological groupoid $\GGG[p] \toto M $ is called the \emph{pull-back of $\GGG \toto B$ with respect to $p: M \rightarrow B $}, or simply the pull-back groupoid when there is no risk of confusion.
\end{defi}

Indeed, the previous functor takes values in the category of Lie groupoids when $p$ is a surjective submersion.
More generally  \cite{CrainicMoerdijk}:

 \begin{lem}\label{lem: generalized surjective submersion}
 Let $\GGG \toto B$ be a Lie groupoid, $M$ be a manifold and $p:M \to B$ a smooth map. Then  $\GGG[p] $ admits a structure of Lie groupoid on the manifold $M$ if the map $\phi : M \times_{p,B,s}\GGG \to B$ given by $(m,\gamma) \mapsto t(\gamma) $, for all $(m,\gamma)\in M \times_{p,B,s}\GGG$, is a surjective submersion (in which case $p$ is called a \emph{generalized surjective submersion for $\GGG \toto B$}).
 \end{lem}
 \begin{proof}
 Lemma \ref{lem:manifold} applied to $\phi : M\times_{p,B,s}\GGG \to B$ and $p: M \to B$ implies that $(M \times_{p,B,s}\GGG) \times_{t,B,p} M$ is a manifold. It is routine to check that $(M \times_{p,B,s}\GGG) \times_{t,B,p} M$, together with the structure maps defined in definition \ref{def:pback} is a Lie groupoid.
 \end{proof}

Given a covering ${\mathcal U}=(U_i)_{i\in I}$ of a manifold $N$, the \emph{\v{C}ech groupoid } is the pull-back of the trivial groupoid $N \toto N$ with respect to the map $ \coprod_{i\in I}U_i \to M$ given by $x_i \mapsto x$ for all $x \in U_i$ (see notations above). Let us give an explicit description of it: the \v{C}ech groupoid is, explicitly, the Lie groupoid $\coprod_{i,j \in J} U_{ij} \toto \coprod_{i\in I} U_i$ with source $s(x_{ij}):=x_i$, target $t(x_{ij}):=x_j$, product $x_{ij}\bullet x_{jk}:=x_{ik} $, the unit map $\epsilon(x_i):=x_{ii} $ and inverse~$x_{ij}^{-1}:=x_{ji}$. In general, we shall simply denote the \v{C}ech groupoid by $\Cechs$ (instead of $ \prod_{i,j \in J} U_{ij} \toto \coprod_{i\in I} U_i$ or $N[\coprod_{i\in I} U_i ] $).

\section{Lie groupoids $\G \to \HH$-extensions}
\label{sec:GpdExt}

Let $\G \to \HH$ be a crossed module of finite dimensional Lie groups. The purpose of this section is to give a complete description, purely in terms of Lie groupoids, of  $\G \to \HH $-gerbes over a given stack, and to check that, when the stack in question is simply a manifold $N$, our notion gives back an already known description \cite{BM,BaezSchreiber,Dedecker} in terms of non-Abelian cohomology.

\subsection{Definition of Lie groupoid $\G \to\HH$-extension}
\label{subsec:defofextensions}

 In \cite{BL}-\cite{LaurentStienonXu} gerbes are described as Lie groupoids extensions (up to Morita equivalence of those). But this description mainly covers the case of the so-called $\G$-gerbes, i.e. gerbes valued in the crossed module $\G \to \Aut (\G)$. In order to describe, in purely Lie groupoid terms, $\G \to \HH$-gerbes, one needs to go further and to work with $\G$-extensions endowed with some principal $\HH$-bundle structure (up to Morita equivalence of those).

We first wish to introduce Lie groupoid extensions.

  \begin{defi}\cite{LaurentStienonXu}
  A Lie groupoid extension is a triple $(\mathcal{R},\GGG,\varphi)$, denoted by $\mathcal{R} \stackrel{\varphi}{\to}\GGG$ (or simply by $\mathcal{R}  \to \GGG$, when there is no risk of confusion), where $\mathcal{R}\toto M$ and $\GGG \toto M$ are Lie groupoids over the (same) manifold $M$ and the map $\varphi: \mathcal{R}  \to \GGG $ is a groupoid morphism over the identity of $M$ such that $\varphi$ is surjective submersion.
  \end{defi}

  The \emph{kernel} of a Lie groupoid extension $\mathcal{R} \stackrel{\varphi}{\to}\GGG$ is, by definition, the inverse image through $\varphi$ of the unit manifold of $\GGG$, \emph{i.e.} the set  $$K=\{r\in \mathcal{R} : \varphi(r) \in \epsilon(M) \}.$$ Since $\varphi$ is a surjective submersion, the kernel is a submanifold of $\mathcal{R}$. Also, since $\varphi$ is a groupoid homomorphism over the identity of $M$, $K$ is indeed a bundle of Lie groups (\emph{i.e.} it is a Lie groupoid whose source and target maps coincide). Notice that $K$ is normal in $\mathcal{R}$ in the sense that $r^{-1} \gpdproductR k \gpdproductR r \in K$ for all admissible $k \in K, r \in \mathcal{R}$.
  The previous assignment defines indeed a Lie groupoid action of $\mathcal{R} $ on $K \to M $, action that we shall denote by $\gpdactkernel $.

  \begin{defi}\cite{LaurentStienonXu}
   Let $\G$ be a Lie group. A Lie groupoid extension $\mathcal{R} \stackrel{\varphi}{\to}\GGG $ is called a Lie groupoid $\G$-extension if its kernel $K$ is locally trivial with typical fiber $\G$,\emph{i.e.} if every point $x\in M$ ($M$ being the base manifold of both $\mathcal{R}$ and $\GGG $) admits a neighborhood $U$ such that $K_{U}=\varphi^{-1}(\epsilon(U))$ is isomorphic to $\G\times U$.
  \end{defi}

  To a Lie groupoid $G$-extension $\mathcal{R} \stackrel{\varphi}{\to}\GGG$, we now associate a principal $\Aut(\G)$-bundle over the groupoid $ \mathcal{R}  \toto M$, called the band of the extension. We first recall the notion of principal $\HH$-bundle over a Lie groupoid. See \cite{LTX} for instance.

  \begin{defi}
  Let $\HH$ be a Lie group and $ \mathcal{R}  \toto M$ be a Lie groupoid.
  A principal $\HH$-bundle over the Lie groupoid $\mathcal{R}\toto M$ is an usual (right) principal $\HH$-bundle $P \stackrel{\pi}{\to} M$  together with a (left) action of the Lie groupoid $ \mathcal{R} \toto M$ on $P \stackrel{\pi}{\to} M$ such that the ${\mathcal R}$ and the $\HH$ actions commute, \emph{i.e.} if we denote the action of the Lie groupoid $\mathcal{R}\toto M$ and the action of the Lie group $\HH$ on $P \stackrel{\pi}{\to} M$, both by the same notation $\cdot$, then $(\gamma\cdot p)\cdot h=\gamma\cdot(p\cdot h)$, for all admissible $\gamma \in \mathcal R, p\in P, h\in \HH$.
  \end{defi}

    We also define morphisms between two principal bundles w.r.t.different groups over different Lie groupoids, as follow.

\begin{defi}\label{defi:isomppalbundles}
A morphism from a principal $\HH$-bundle $P \stackrel{\pi}{\to} M$ over a Lie groupoid $\mathcal{R}\toto M$ to a  principal $\HH'$-bundle $P' \stackrel{\pi'}{\to} M'$ over a Lie groupoid $\mathcal{R'}\toto M'$ is triple a $(\Phi,\Psi,\jmath)$, where $\Phi : \mathcal R \rightarrow \mathcal R'$ is an morphism of Lie groupoids $\Psi: P \to P' $ is diffeomorphism  and $\jmath: \HH \to \HH' $ be a Lie group morphism, such that:
    $$ \Psi ( \gamma \gpdactP p \cdot h ) = \Phi(\gamma) \, \bullet_{{\mathcal R'},{P'}} \,  \Psi (p) \cdot \jmath(h)   $$
    for all pair $(\gamma,p) \in \mathcal{R} \times_{t,M,\pi} P $ and all $h \in \HH$.
%
When the Lie groupoids $\mathcal{R}\toto M$ and $\mathcal{R'}\toto M'$ are identically, the same Lie groupoid and the map $\Phi $ is identity, then the morphism $(\Phi,\Psi,\jmath)$ is called  \emph{morphism over the identity of ${\mathcal R} \toto M $} and  simply denoted by pair $(\Psi,\jmath)$.
  \end{defi}
The band \cite{Giraud,BM} is, in general, defined for gerbe itself, but \cite{LaurentStienonXu} introduced a notion of band for Lie groupoid  $\G$-extensions that boils down to the band of  gerbes. Let $\G$ be a Lie group. Then \emph{ band of a given $\G$-extension $\mathcal{R}\stackrel{\varphi}{\rightarrow} \GGG$}, by construction, is the set of all Lie group morphisms from $\G$ to some fiber of its kernel. More precisely, let $\G$ be a Lie group and $\mathcal{R} \stackrel{\varphi}{\rightarrow} \GGG$ be a $\G$-extension, we set
\begin{equation}
Band( \mathcal{R} \rightarrow \GGG ) :=\coprod_{m\in M} isom(\G,K_{m})
\end{equation}
to be the set of all possible Lie group isomorphisms from $\G$ to fiber $K_m$, for some $m\in M$, where $K_m=\{k\in K| \varphi(k)=\epsilon(m)\}$.
Recall from \cite{LaurentStienonXu} that
\begin{enumerate}
\item \label{actAutonBand}$Band( \mathcal{R} \rightarrow \GGG ) $ admits a natural manifold structure, for which the projection on $M$ is a smooth surjective submersion.
We let $Band_m( \mathcal{R} \rightarrow \GGG )$ stand for the fiber over $m \in M $.
\item $\Aut (\G)$ acts (on the right) freely and transitively on the fibers of $Band( \mathcal{R} \rightarrow \GGG ) $ as follows: $b_{m}\cdot \rho:=b_{m}\circ \rho$, for all $\rho \in \Aut (\G)$ and  $b_{m}\in Band_m( \mathcal{R} \rightarrow \GGG)$.
\end{enumerate}
All these items together imply that for a given Lie group $\G$ and a $\G$-extension $\mathcal{R}\to \GGG$, $Band( \mathcal{R} \rightarrow \GGG ) \stackrel{\pi}{\to} M$ is a (right) principal $\Aut(\G)$-bundle over the base manifold $M$, where $\pi$ is the obvious projection to the manifold $M$. Moreover, $Band( \mathcal{R} \rightarrow \GGG ) \stackrel{\pi}{\to} M$ is a principal $\Aut(\G)$-bundle over the Lie groupoid $\mathcal{R}\toto M$, when equipped with a left action of $\mathcal{R}\toto M$ on $Band( \mathcal{R} \rightarrow \GGG ) \stackrel{\pi}{\to} M$, defined by setting $r\gpdactband b_{m} $ to be the Lie group morphism from $\G$ to $K_{s(r)} $, given by
 \begin{equation}\label{def:def of gpdactband}
 g \mapsto r b_m(g) r^{-1},
 \end{equation}
 For all $r\in \mathcal{R}$ with $t(r)=m,b_m\in isom(\G,K_m)$.


We now have all the tools required for defining the type of extension whose (to be defined in section \ref{sec:MoritaExt}) quotients shall define $\G \to \HH $-gerbes.

\begin{defi}\label{def: Azimi-extension}
Let $\G \stackrel{\rho}{\to} \HH$ be a crossed module, with action map $\jmath: \mathbf H \to \Aut (\G)$ and $\GGG \toto M $ be a Lie groupoid. A Lie groupoid $\G \to \mathbf H$-extension(or simply a $\G \to \mathbf H$-extension) of the Lie groupoid $\GGG \toto M $ is a triple $(\mathcal{R} \rightarrow \GGG, P \to M, \chi)$, where:
 \begin{enumerate}
 \item $\mathcal{R} \rightarrow \GGG$ is a Lie groupoid $\G$-extension,
 \item  $P \to M $ is a principal $\HH$-bundle over the Lie groupoid $ \mathcal{R} \toto M$,
 \item $(\chi,\jmath)$ is a morphism over the identity of $ {\mathcal R} \toto M$ (see definition \ref{defi:isomppalbundles}) from the principal $\HH$-bundle $P \to M $ to the principal $\Aut(\G)$-bundle $Band( \mathcal{R} \rightarrow \GGG ) $,
 \end{enumerate}
      such that, for all $p\in P, g \in \G$:
  \begin{equation}
  \label{eq:chi} \begin{array}{rcl}
    p \cdot \rho (g) & =  & \chi(p) (g) \gpdactP p  \end{array}
    \end{equation}
    (recall that $\chi(p)$ belongs to $Band_{\pi(p)}(  \mathcal{R} \rightarrow \GGG)= Aut (\G, K_{\pi(p)}) $, so that  $\chi(p) (g)$ is an element in $ K_{\pi (p)} \subset {\mathcal{R}}$: it makes therefore sense to let it act on $p \in P$).
\end{defi}

It shall be convenient to draw the following diagram in order to represent $\G \to \HH$-extensions.
Below, it shall be understood that an arrow of the type $\mathcal{R} \xymatrix{  \ar@{{}{--}^{)}}[r] & P}$ means that
the groupoid $\mathcal{R} $ acts on $P$.
$$\xymatrix{  {\mathcal R} \ar@{{}{--}^{)}}[dr] \ar@{{}{--}^{)}}[drr]  \ar[d]  &   \\ {\GGG} \ar@<2pt>[d] \ar@<-2pt>[d] & \ar[dl] P \ar[r]^-{\chi }  & \ar[dll] Band \\  M\\} $$

Let $\G \to \HH$ be a crossed module. By an \emph{isomorphism between two $\G \to \HH$-extensions of Lie groupoid  $\GGG \toto M $}, namely
$(\mathcal{R} \rightarrow \GGG, P \to M, \chi)$
and
$(\mathcal{R}' \rightarrow \GGG, P' \to M, \chi')$, we mean an isomorphism $(\Phi,\Psi,id_\HH)$ of principal bundle over Lie groupoids (see definition \ref{defi:isomppalbundles}) such that the following diagram commutes:
\begin{equation}\label{diagram for isomorphism} \begin{CD}
P @>\Psi>> P'\\
@VV\chi V @VV\chi'V\\
Band(\mathcal{R}\rightarrow\GGG ) @>\bar{\Phi}>> Band(\mathcal{R'}\rightarrow\GGG)
\end{CD}
\end{equation}
where $\bar{\Phi}(\eta)(g)=\Phi(\eta(g))$, for $\eta \in Band(\mathcal{R}\rightarrow\GGG ), g \in \G $. For the sake of simplicity we suppress $id_\HH$ and use the notation $(\Phi,\Psi)$ for such an isomorphism.




\begin{examp}
\label{ex:1toHH}
When the crossed module is simply $ \{1\} \to \HH $, then $\G \to \HH$-extensions are nothing than principal $\HH$-bundles over Lie groupoids, and isomorphisms of $\G \to \HH$-extensions amount to isomorphisms of those.
\end{examp}

\begin{examp}\label{ex:G_extensions_as_Azimi_extensions}
For every $\G$-extension $\mathcal{R} \rightarrow \GGG \toto M$, the quadruple
 \begin{equation}\label{eq:quadruple} (\mathcal{R} \rightarrow \GGG , Band(\mathcal{R} \rightarrow \GGG  ) \to M, \id_{Band(\mathcal{R} \rightarrow \GGG )}) \end{equation}
 is a $\G \to \Aut (\G)$-extension.
Conversely, when the crossed module $\G \to \HH$ is $ \G \to \Aut (\G) $, then for every $\G \to \HH$-extension $(\mathcal{R} \rightarrow \GGG , P \to M, \chi)$, the pair $(\chi,Id_{\Aut (\G)})  $ is an isomorphism of principal bundles over the identity of  $\mathcal{R} \toto M$.
  In conclusion, the assignment of (\ref{eq:quadruple}) induces a one-to-one correspondence between $ \G \to \Aut (\G) $-extensions and $\G$-extensions. This correspondence is an equivalence of categories, for isomorphisms of $\G \to \Aut(\G)$-extension amount to isomorphisms of the corresponding $\G$-extension.
\end{examp}
\begin{rem}
\label{rem:linkGinotStienon}
The referee pointed to us the next point, relating our construction with \cite{GinotStienon}.
As recalled in \cite{GinotStienon}, crossed module of Lie groups $\G\stackrel{\rho}{\to}\HH$ can be seen as a Lie $2$-group. To say it briefly, the $2$-arrows are all triples $(h_1,g,h_2) \in \HH \times \G \times \HH$ subjects to the constraint
 $$ h_1  = \rho(g)h_2 .$$
The horizontal and vertical products are given by:
 $$ (h_1,g_1,h_2) ._V (h_2,g_2,h_3) = (h_1,g_1g_2,h_3) \hbox{ and } (h_1,g_1,h_2)._H(h_3,g_2,h_4)=(h_1h_3, g_1h_2(g_2),h_2h_4) $$
respectively.

Now, let $(\mathcal{R}\stackrel{\varphi}{\to}\GGG, P\stackrel{\pi}{\to}M,\chi)$ be a $\G\stackrel{\rho}{\to}\HH$-extension of the Lie groupoid $\GGG\toto M$,
 with kernel $K$. There is a natural $2$-groupoid with $2$-arrows the set $ \mathcal{R} \times_\GGG \mathcal{R}$ of all pairs of elements in $\mathcal{R}$ projecting over the same element of $\GGG$: the horizontal and vertical products $._V$ and $._H$ being given in the obvious manner:
  $$ (r_1,r_2) ._V(r_2,r_3)=(r_1,r_3) \hbox{ and } (r_1,r_2) ._H(r_3,r_4) = (r_1r_3,r_2r_4)  .$$

Assume that there exists a section $\sigma : M \to P$ of the projection $ \pi: P \to M $.
A natural identification of the kernel $K$ with $\G \times M$ is induced: more precisely, $\chi \circ \sigma $ is,  at all point $m \in M$, a Lie group isomorphism between $K_m$ and $\G$.
Since $r\bullet \sigma(t(r))$ and $\sigma(s(r))$ are in same fiber of $\pi: P \to M$,  there is an unique element $\psi(r)\in \HH$ such that $r\bullet\sigma_{t(r)}=\sigma_{s(r)}\cdot\psi(r)$. The map
\begin{equation*}
\begin{array}{c}
\psi:\mathcal{R}\to \HH\\
r \bullet \sigma (t(r))=\sigma (s(r)) \cdot \psi(r),
\end{array}
\end{equation*}
where $\cdot$ stands for the action of Lie group $\HH$ on the manifold $P$, is well-defined.
A direct verification shows that sending a pair $(r_1,r_2) $ of elements in $ \mathcal{R} \times_\GGG \mathcal{R} $
to the triple $ ( \psi(r_1), \chi \circ \sigma^{-1} (r_1r_2^{-1})  , \psi(r_2) ) \in \HH \times \G \times \HH$, one obtains a morphism of Lie $2$-groupoid.

As a consequence, for every  $\G\stackrel{\rho}{\to}\HH$-extension $(\mathcal{R}\stackrel{\varphi}{\to}\GGG, P\stackrel{\pi}{\to}M,\chi)$,
a morphism of $2$-groupoid from  $ \mathcal{R} \times_\GGG \mathcal{R} $ to the crossed module can be constructed.
Now, it happens that $ \mathcal{R} \times_\GGG \mathcal{R} $ is Morita equivalent, in a sense defined in \cite{GinotStienon}, to $\GGG $, so that
 for every  $\G\stackrel{\rho}{\to}\HH$-extension $(\mathcal{R}\stackrel{\varphi}{\to}\GGG, P\stackrel{\pi}{\to}M,\chi)$ a $\G \to \HH$-extension
in the sense of \cite{GinotStienon} can be defined.

This construction can be done backward, but there is a delicate point.
A $2$-group bundle in the sense of \cite{GinotStienon} over a Lie groupoid $\GGG$ will not induce in general a
$\G \to \HH$-extension of $\GGG$, but a $\G \to \HH$-extension of a Lie groupoid $\GGG'$ which is a pull-back of $\GGG$.
\end{rem}

\subsection{The manifold case: $\G \to \HH $-valued non-abelian cocycles as $\G \to \HH$-extensions over  Lie groupoid}
\label{subsec:mfdcase}

Throughout the present section, we shall fix an open covering ${\mathcal U}= (U_i)_{i \in I}$ of a manifold~$N$.

 Our purpose is to show that $\G \to \HH$-extensions of the  \v{C}ech groupoid $\Cechs$ correspond to $\G \to \HH$-valued non-Abelian $1$-cohomology, computed with respect to ${\mathcal U}$. We first recall the notion of non-Abelian $1$-cocycles \cite{Dedecker,BM}, as introduced by Dedecker. Then, we show that these are in one-to-one correspondence with (a certain set of) $\G \to \HH$-extensions of the  \v{C}ech groupoid $\Cechs$. Proving that $\G \to \HH$-coboundaries correspond to isomorphisms of these extensions shall then yield to the desired conclusion.

\begin{defi} \label{definition of Adapted extension}
An adapted Lie groupoid $\G \to \HH$-extension of the \v{C}ech groupoid $\Cechs $ is a $\G \to \HH$-extension  $({\mathcal R} \stackrel{\varphi}{\to}  \Cechs, P \to \coprod_{i \in I} U_{i}, \chi) $
on which we impose the following constraints:
\begin{enumerate}
  \item ${\mathcal R}$ is the space $ \G \times  \coprod_{i,j \in I} U_{ij}$ and $ \varphi$ is the projection onto the second component,
    \item $P$ is the space $  \coprod_{i \in I} U_{i} \times \HH $, equipped with the trivial right $\HH$-action $ (x_i,h) \cdot h'= (x_i,hh')$ for all $h,h' \in \HH, x\in U_{i}$,
          \item The map $\chi : P \to Band ( {\mathcal R} \stackrel{\phi}{\to}  \Cechs)$ maps $ (x_i,h) \in P $ to the element of the band over $x_i$ given by $  g \mapsto (h (g),x_{ii}) $ for all $g \in \G, h \in \HH, x \in U_i $, where we have used the same notation for
   \item the Lie groupoid product $\gpdproductR$ of ${\mathcal R}$ satisfies the relation $(g,x_{ii}) \gpdproductR (g',x_{ij}) =
     (gg', x_{ij}) $ for all~$x \in U_{ij}, g,g' \in \G, i,j\in I$.
    \end{enumerate}
\end{defi}
  Items 1 and 4 of the definition imply that the kernel of an adapted Lie groupoid $\G \to \HH$-extension of the \v{C}ech groupoid $\Cechs $ is the trivial bundle of group:
$K=\G \times \coprod_{i \in I} U_{ii}\simeq \G \times \coprod_{i, \in I} U_{i}$
so that the band $Band(\mathcal R \rightarrow \Cechs)$ is canonically isomorphic to
 $\Aut(\G) \times \coprod_{i \in I} U_{i} $.

For a given manifold $N$, a given open covering $ {\mathcal U}$ and a given crossed module $\G \to \HH $, adapted Lie groupoid $\G \to \HH$-extensions of the same \v{C}ech groupoid $\Cechs $ may only differ by two things, namely the Lie groupoid product on  $\mathcal{R}=\G \times \coprod_{i,j \in I} U_{ij}$ and the action of $ \mathcal{R}=\G \times \coprod_{i,j \in I} U_{ij}$ on $P:= \coprod_{i \in I} U_i \times \G $.

\begin{nota}\label{not:adaptextensions}
For $ {\mathcal U}$ an open covering of a manifold $N$, we shall denote adapted Lie groupoid $\G \to \HH$-extensions of $\Cechs $ as triples $({\mathcal U},\bullet, \star)$, where $\mathcal U$ refers to the open covering, $\bullet $ refers to the multiplication of the Lie groupoid $\mathcal{R}:= \G \times \coprod_{i,j \in I} U_{ij}$ and $\star$ refers to the action of $\mathcal{R}:= \G \times \coprod_{i,j \in I} U_{ij}$ on the principal bundle $P:=  \coprod_{i \in I} U_i \times \G $.
\end{nota}


\begin{rem}\label{rmk:remark_on_adapted}
For an adapted Lie groupoid $\G \to \HH$-extension $(\mathcal U,\bullet, \star)$, the action of an element $(g, x_{ii}) $  in the kernel of ${\mathcal R} \stackrel{\phi}{\to}  \Cechs$ on an admissible element $(x_i,h) \in P$ is given by $(x_i,\rho(g)h)$. We prove it as follows.
First notice that:
\begin{equation}
\begin{array}{rcll}
(x_{i},h) \cdot \rho(g) & = & (x_{i},h\rho(g)) & \hbox{by definition \ref{definition of Adapted extension}, item 2}\\
                 & = & (x_{i},h\rho(g)h^{-1}h) & \hbox{}\\
                 & = & (x_{i},\rho(h(g))h) & \hbox{by axioms of crossed module.}
\end{array}
\end{equation}
On the other hand:
\begin{equation}
\begin{array}{rcll}
(x_{i},h) \cdot \rho(g)
                 & = & \chi (x_{i}, h)(g) \star (x_{i},h) &\hbox{by (\ref{eq:chi}) in definition \ref{def: Azimi-extension}}\\
                 & = & (h(g),x_{ii}) \star (x_{i},h) & \hbox{by definition \ref{definition of Adapted extension}, item 3.}
\end{array}
\end{equation}
The result follows by substituting $h(g) $ by $g $ in the previous relations.
\end{rem}
We now recall from \cite{Dedecker} the notion of non-Abelian $1$-cocycles valued in an arbitrary crossed module $\G \to \HH$. We use the notation $e$ for the neutral element of both Lie groups $\G, \HH$.
\begin{defi}\label{def:nonab}
Let $\G\stackrel{\rho}{\to} \HH$ be a crossed module, and ${\mathcal U}=(U_i)_{i \in I} $ an open covering of a manifold $N$.
A non-Abelian $1$-cocycle w.r.t. ${\mathcal U} $ with values in $\G\to \HH$ is a pair $(\lambda, \gggg)  \in \mathcal C ^{\infty}( \coprod_{i,j \in I}U_{ij},\HH ) \times \mathcal C ^{\infty} (\coprod_{i,j,k \in J}U_{ijk},\G) $
 required to satisfy the following conditions:
   \begin{equation}\label{nonAbe}  \left\{
    \begin{array}{l}
   \rho(\gggg_{ijk})  \lambda_{ik} =  \lambda_{ij}\lambda_{jk} \\
    \gggg_{ijk}\gggg_{ikl}=\lambda_{ij}(\gggg_{jkl})\gggg_{ijl}\\
    \gggg_{iij}=e
\end{array}   \right. \end{equation}
for all possible indices (here $\lambda_{ij}$ (resp.$\gggg_{ijk}$) stands for the restriction of $\lambda$ (resp. $\gggg$ ) to $ U_{ij} $ (resp.$ U_{ijk} $ )).
\end{defi}

\begin{rem} \label{rmk:lambdaii} Note that the first relation in  (\ref{nonAbe}), when $ i=j$, implies that $\lambda_{ii}=\e $, for all $i\in I$.
\end{rem}

We now prove the desired correspondence which generalizes \cite{BL} (recall that $\jmath: \HH \to \Aut (\G)$ is part of the crossed module structure, see section \ref{sec:prerequisites}).

\begin{prop}\label{prop:cocycles}
Let $ {\mathcal U}=(U_i)_{i \in I}$ be an open covering of a manifold $N$,
and $\G \to \HH$ a crossed module of Lie groups.
\begin{enumerate}
\item Let $(\mathcal U,\bullet, \star)$ be an adapted Lie groupoid $\G \to \HH$-extension of \v{C}ech groupoid $N[\mathcal U]$. We define $(\lambda,\gggg) \in  \mathcal C ^{\infty}( \coprod_{i,j \in I}V_{ij},\HH ) \times \mathcal C ^{\infty} (
\coprod_{i,j,k \in J}V_{ijk},\G) $ gluing together the family of maps $\lambda_{ij}:U_{ij} \to \HH$ and $\gggg_{ijk}:U_{ijk} \to \G $ defined by
    \begin{equation}
  \label{eq:cons}
      \left\{ \begin{array}{rcll}

                        (\e,x_{ij}) \star (x_{j},\e) &=& ( x_{i}, \lambda_{ij} ) & \forall i,j\in I, \forall x\in U_{ij} \\
                      (\e,x_{ij}) \bullet  (\e,x_{jk}) &=& (\gggg_{ijk}, x_{ik}) & \forall i,j,k\in I \forall x\in U_{ijk} .\\
       \end{array} \right.
  \end{equation}
  Then $(\lambda_, \gggg)$ is a non-Abelian $1$-cocycle.


 \item Given a non-Abelian $1$-cocycle  $(\lambda ,\gggg)$, we define:
   \begin{enumerate}
   \item a Lie groupoid structure $\bullet$ on $\mathcal R =\G \times \coprod_{i,j\in I} U_{ij} $ by:
     \begin{equation}\label{eq:recons}  \left\{ \begin{array}{rcl} (g,x_{ij})  (g',x_{jk}) &:=&
     (g\lambda_{ij}(g')\gggg_{ijk}, x_{ik}) \\ (g,x_{ij})^{-1} &:=& (
     \lambda_{ij}^{-1}(g^{-1}\gggg_{iji}^{-1}) , x_{ji})
 \end{array} \right. \end{equation}
 for all $g,g' \in \G, i,j \in I, x \in U_{ij}$, where $\lambda_{ij}(g')$ is a short hand notation for $\jmath (\lambda_{ij})(g')$.
   \item a map $ \phi: \mathcal R \to N[\mathcal U] $ given by $ (g,x_{ij}) \mapsto x_{ij} $, for all $ g \in \G ,i,j \in I , x \in U_{ij} $,
   \item a structure of principal $\HH$-bundle $\star$ on $P:=\coprod_{i\in I} U_{i} \times \HH $ over the Lie groupoid $\mathcal R \toto \coprod_{i\in I} U_{i} $ by
  \begin{equation}\label{defi: ppalstronP}
(g,x_{ij}) \star (x_{j},h)=(x_{i},\rho(g)\lambda_{ij} h),
\end{equation}
for all $g\in \G , h\in \HH, i,j\in I, x\in U_{ij}$,
   \item a map $\chi: P \to Band(\mathcal R \to N[\mathcal U])$ by $(x_{i},h)\mapsto (\jmath(h),x_{i}),$
   for all $ h\in \HH, i,\in I, x\in U_{i}$,
   \end{enumerate}
   then $(\mathcal U,\bullet , \star )$ is an adapted Lie groupoid $\G \to \HH$-extension of the \v{C}ech groupoid $N[\mathcal U]$,

   \item the procedures in items 1 and 2 are inverse to each other.


 \end{enumerate}
 \end{prop}

The proof will go through a lemma.

\begin{lem} \label{key lemma}
Let $(\mathcal U,\bullet, \star) $ be an adapted Lie groupoid $\G\to\HH$-extension of the \v{C}ech groupoid $\Cechs $. Define the maps $\lambda_{ij}:U_{ij} \to \HH$ and $\gggg_{ijk}:U_{ijk} \to \G $ as in (\ref{eq:cons}). Then the following relation holds for all $i,j \in I, g\in \G$ and $x \in U_{ij}$:
 $$(e,x_{ij}) \bullet (g,x_{jj})=(\lambda_{ij}(g), x_{ij})$$
\end{lem}
\begin{proof}
First observe that for all $i,j \in I, g\in \G$ and $x \in U_{ij}$ we have
\begin{equation*}
\begin{array}{rcll}
\chi((e,x_{ij}) \star (x_{j},e))(g) &=& \chi(x_{i},\lambda_{ij})(g) & \hbox{by (\ref{eq:cons}), i.e. definition of $\lambda_{ij}$}\\
                                  &=& (\lambda_{ij}(g),x_{ii}) & \hbox{by definition \ref{definition of Adapted extension} item 3}.
\end{array}
\end{equation*}
On the other hand,
\begin{equation*}
\begin{array}{rcll}
                                 && \chi((e,x_{ij})\star (x_{j},e))(g)\\
                                  &=& ((e,x_{ij}) \gpdactband \chi(x_{j},e)) \, (g) & \hbox{$\chi$ is a morphism of ppal bundles over grpds}\\
                                  &=& (e,x_{ij}) \bullet ( g,x_{jj})\bullet (e,x_{ij})^{-1}, &\hbox{by def of $\gpdactband$ i.e. \ref{actAutonBand} }
\end{array}
\end{equation*}
for all $i,j \in I, g\in \G$ and $x \in U_{ij}$.
%
Multiplying on the right of both sides of the last two relations by $ (e,x_{ij}) $ and using item $4$ in definition \ref{definition of Adapted extension} yield the desired relation.
%
\end{proof}

\begin{proof}(of proposition \ref{prop:cocycles}).\emph{1}) We first prove that the maps defined in item $1$ form a non-Abelian $1$-cocycle. The relation $\gggg_{ijj}=e$ is obtained by putting $j=k$ in the second relation of (\ref{eq:cons}).
 By definition of groupoid action, we have
 \begin{equation}\label{associativity}
 ((e,x_{ij})\bullet (e,x_{jk}))\star (x_{k},e)=(e,x_{ij})\star((e,x_{jk})\star(x_ k,e))
\end{equation}
for all indices $i,j,k$ and all $x\in U_{ijk}$. The LHS of (\ref{associativity}) gives:
\begin{equation*}
\begin{array}{rcll}\label{LHS}
\hbox{LHS of} \,(\ref{associativity}) & = & (\gggg_{ijk},x_{ik})\star(x_{k},e) & \hbox{by (\ref{eq:cons}), i.e. def. of $\gggg_{ijk}$}              \\
                           & = & ((\gggg_{ijk},x_{ii}) \bullet (e,x_{ik}))\star(x_{k},e)  & \hbox{by def. \ref{definition of Adapted extension}, item 4}\\
                            & = & (\gggg_{ijk},x_{ii}) \star ((e,x_{ik})\star(x_{k},e))  & \hbox{by axioms of groupoid action}\\
                           & = & (\gggg_{ijk},x_{ii}) \star (x_{i},\lambda_{ik}) &\hbox{by (\ref{eq:cons}), i.e. def. of $\lambda_{ik}$}  \\
                           & = & (x_{i},\rho(\gggg_{ijk})\lambda_{ik}) & \hbox{by remark \ref{rmk:remark_on_adapted}}
\end{array}
\end{equation*}
while the RHS of (\ref{associativity}) gives:
\begin{equation*}
\begin{array}{rcll}\label{RHS}
\hbox{RHS of}\, (\ref{associativity}) & = &(e,x_{ij})\star (x_{j},\lambda_{jk}) & \hbox{by (\ref{eq:cons}), i.e. def. of $\lambda_{jk}$} \\
 & = & (e,x_{ij})\star(x_{j},e)  \lambda_{jk} & \hbox{by def. \ref{definition of Adapted extension}, item 2,} \\
                                       & = & (x_{i},\lambda_{ij})\lambda_{jk} & \hbox{by def. of $\lambda_{ij}$} \\
                                       & = & (x_{i},\lambda_{ij}\lambda_{jk}) & \hbox{by (\ref{eq:cons}), i.e. def. \ref{definition of Adapted extension}, item 2}
\end{array}
\end{equation*}
Comparing these relations, we obtain the first condition of (\ref{nonAbe}).
 To show that the henceforth constructed families $(\lambda_{ij})_{i,j \in I}$ and $(\gggg_{ijk})_{i,j,k \in I}$ satisfy the second condition of (\ref{nonAbe}), we write the associativity condition of the Lie groupoid multiplication of $\mathcal{R} $ as follows:
\begin{equation}\label{associativity 2}
((e,x_{ij})\bullet(e,x_{jk}))\bullet(e,x_{kl})=(e,x_{ij})\bullet((e,x_{jk})\bullet(e,x_{kl}))
\end{equation}
for all indices $i,j,k,l \in I$ and $x\in U_{ijkl}$. The LHS of (\ref{associativity 2}) amounts to:
\begin{equation*}
\begin{array}{rcll}
\hbox{LHS of } (\ref{associativity 2})& = & (\gggg_{ijk},x_{ik})\bullet(e,x_{kl}) & \hbox{by (\ref{eq:cons}), i.e. definition of $\gggg_{ijk}$}\\
                                     & = & ((\gggg_{ijk},x_{ii})\bullet(e,x_{ik}))\bullet(e,x_{kl}) & \hbox{by definition \ref{definition of Adapted extension}, item 4}\\
                                     & = & (\gggg_{ijk},x_{ii})\bullet((e,x_{ik})\bullet(e,x_{kl})) & \hbox{(by associativity of the gpd product)}\\
                                     & = & (\gggg_{ijk},x_{ii})\bullet(\gggg_{ikl},x_{il})   & \hbox{by (\ref{eq:cons}), i.e. definition of $\gggg_{ikl}$ }\\
                                     & = & (\gggg_{ijk}\gggg_{ikl},x_{il})            & \hbox{by definition \ref{definition of Adapted extension}, item 4}
\end{array}
\end{equation*}
while the RHS of (\ref{associativity 2}) gives
 \begin{equation*}
 \begin{array}{rcll}
  \hbox{RHS of } (\ref{associativity 2}) & = & (e,x_{ij})\bullet(\gggg_{jkl},x_{jl}) & \hbox{by  (\ref{eq:cons}), i.e. definition of $\gggg_{jkl}$ } \\
                                       & = & (e,x_{ij})\bullet(\gggg_{jkl},x_{jj}) \bullet (e,x_{jl})  &\hbox{by definition \ref{definition of Adapted extension}, item 4}\\
                                       & = & (\lambda_{ij}(\gggg_{jkl}),x_{ij}) \bullet (e,x_{jl})               & \hbox{by lemma \ref{key lemma}}\\
                                       & = & (\lambda_{ij}(\gggg_{jkl}),x_{ii}) \bullet (e,x_{ij})\bullet(e,x_{jl})     &\hbox{by definition \ref{definition of Adapted extension}, item 4}\\
                                       & = & (\lambda_{ij}(\gggg_{jkl}),x_{ii}) \bullet (\gggg_{ijl},x_{il})     & \hbox{by (\ref{eq:cons}), i.e. definition of $\gggg_{ijl}$ }\\
                                       & = & (\lambda_{ij}(\gggg_{jkl})\gggg_{ijl},x_{il})            & \hbox{by definition \ref{definition of Adapted extension}, item 4}
 \end{array}
 \end{equation*}
Comparing these relations, we obtain the second condition of (\ref{nonAbe}), which completes the proof of the first item.

\bigskip
\noindent
\emph{2}   We need to check that the multiplication\, $\bullet$ defined in (\ref{eq:recons}) is a Lie groupoid multiplication. We first prove the associativity:
For all $i,j,k \in I, x \in U_{ijkl}$, $g,g',g'' \in \G$, we compute:
\begin{equation*}
\begin{array}{rcll} & & (g,x_{ij})\bullet ((g',x_{jk}) \bullet(g'',x_{kl})) & \\
 & = & (g,x_{ij}) \bullet (g'\lambda_{jk}( g'')\gggg_{jkl},x_{jl})   &  \hbox{by (\ref{eq:recons}), i.e. def. of $\bullet$}\\
                                    & = & (g\lambda_{ij}(g'\lambda_{jk}( g'')\gggg_{jkl})\gggg_{ijl}, x_{il}) & \hbox{by (\ref{eq:recons}), i.e. def. of$\bullet$}\\
                                    & = & (g\lambda_{ij}( g')\lambda_{ij}(\lambda_{jk}( g''))\lambda_{ij}(\gggg_{jkl})\gggg_{ijl}, x_{il}) & \hbox{by crossed modules axioms}\\
                                    & = & (g\lambda_{ij}( g')(\rho(\gggg_{ijk})\lambda_{ik}( g''))\lambda_{ij} (\gggg_{jkl})\gggg_{ijl},x_{il}) & \hbox{by (\ref{nonAbe}) in definition \ref{def:nonab}}\\
                                    & = & (g\lambda_{ij}( g')\gggg_{ijk}\lambda_{ik}(g'')\gggg_{ijk}^{-1}\lambda_{ij} (\gggg_{jkl})\gggg_{ijl},x_{il}) &\hbox{by crossed module axioms}\\
                                    & = & (g\lambda_{ij} (g')\gggg_{ijk}\lambda_{ik}(g'') \gggg_{ikl},x_{il}) & \hbox{by (\ref{nonAbe}) in definition \ref{def:nonab}}\\
                                    & = & (g\lambda_{ij}( g')\gggg_{ijk},x_{ik})\bullet(g'',x_{kl}) & \hbox{by (\ref{eq:recons}), i.e. def. of $\bullet$}\\
                                    & = & ((g,x_{ij}) \bullet (g',x_{jk})) \bullet (g'',x_{kl}) & \hbox{by (\ref{eq:recons}), i.e. def. of $\bullet$}.
\end{array}
\end{equation*}
 It is routine to check that the henceforth defined multiplication admits as its source map $s$ (resp. its target map $t$) the map $(g,x_{ij})\mapsto x_i$ (resp $x_j$). Also, this multiplication admits the map $\epsilon : \coprod_{i \in I} U_i \to \mathcal R$ given by $x_i \mapsto (e,x_{ii})$ as its unit map, and its inverse given as in (\ref{eq:recons}). Altogether, these structural maps endow $\mathcal R$ with a structure of Lie groupoids, and eventually turn $\mathcal R\stackrel{\phi}{\to} \coprod_{i \in I} U_{ij}$ into a Lie groupoid $\G$-extension. It is also routine to check that (\ref{defi: ppalstronP}) gives a structure of principal $\HH$-bundle over the Lie groupoid $ \mathcal R \toto \coprod_{i\in I} U_{i} $. In order to check that
 $$(\mathcal R \to N[\mathcal U], P \to \coprod_{i\in I} U_{i}, \chi)$$
 is a $\G \to \HH$-extension, we are left with the task of showing that $(\chi, \jmath)$ is a morphism of principal bundles over the identity of $\mathcal R$.
 One condition is obvious:
 $$ \chi((x_i,h)\cdot h')=\chi(x_i,hh')=(\jmath(hh'),x_i)=(\jmath(h)\jmath(h'),x_i)=(\jmath(h),x_i)\jmath(h')=\chi(x_i,h)\jmath(h')$$
  while the following proves that $p \cdot \rho(g)= \chi(p)(g)\star p$ for all $p\in P, g\in \G$, hence proves the claim:

 \begin{equation*}
 \begin{array}{rcll}
                                       & &\chi((x_i,h)\cdot h')(g)\star (x_i,h)&\hbox{}\\
                                       & = & (\jmath(h)(g),x_{ii})\star (x_i,h)& \hbox{by def. of $\chi$}\\
                                       & = & (h(g),x_{ii})\star (x_i,h) & \hbox{}\\
                                       & = & (x_i,\rho(h(g))\lambda_{ii}h) & \hbox{by (\ref{defi: ppalstronP}})\\
                                       & = & (x_i,h\rho(g)h^{-1}h)& \hbox{by crossed module axiom}\\
                                       & = & (x_i,h)\cdot \rho(g).&
 \end{array}
 \end{equation*}
 Now items 1-3 of definition \ref{definition of Adapted extension} hold by construction and item 4 holds because $\gggg_{iij}$ is assumed to be equal to the neutral element $\e$ of $\G$ in definition \ref{def:nonab}. This completes the proof of the second item.


  \bigskip
\noindent
\emph{3})
   Next, we prove that items 1 and 2 in the proposition yield constructions which are inverse one to the other. For this purpose, we first notice that (\ref{eq:recons}) and (\ref{defi: ppalstronP}) hold for any adapted Lie groupoid $\G \to \HH$-extension, hence the construction of item 2 is injective. Assume that we are given a $\G \to \HH$-valued non-Abelian $1$-cocycle $(\lambda_{ij}, \gggg_{ijk})_{i,j,k \in I}$, then applying the procedure in item 2 we obtain an adapted Lie groupoid $\G \to \HH$-extension, to which we apply the construction in item $1$ to yield a  $\G \to \HH$-valued non-Abelian $1$-cocycle $(\lambda'_{ij}, \gggg'_{ijk})_{i,j,k \in I}$. We need to show that these two non-Abelian $1$-cocycles are equal. For this, observe that, by construction in item 2, we have $(x_i, \lambda'_{ij})= (e,x_{ij}) \star (x_j, e)$ while it follows from item 1 that $(e,x_{ij}) \star (x_j, e)=(x_i,\rho(e)\lambda_{ij}e)$. These two relations together prove that $\lambda_{ij}=\lambda'_{ij}$ for all $i,j \in I$. A
similar argument proves that $\gggg_{ijk}=\gggg'_{ijk}$, hence the claim. This implies that if two adapted Lie groupoid $\G \to \HH$-extensions of the \v{C}ech groupoid have the same $\G \to \HH$-valued non-Abelian $1$-cocycles associated with, they are equal. This proves the claim.
\end{proof}

Having made explicit a one-to-one correspondence between adapted $\G \to \HH$-extensions and non-Abelian $1$-cocycles, we now prove that, under this correspondence, isomorphisms of adapted $\G \to \HH$-extensions correspond to non-Abelian coboundaries, a notion that we now introduce, following \cite{BM},\cite{Dedecker}.

\begin{defi}\label{defi:cobound}
Let $ {\mathcal U}=(U_i)_{i \in I}$ be an open covering of a manifold $N$
and $\G \to \HH$ be a crossed module of Lie groups.
 A \emph{$\G \to \HH$-valued $1$-coboundary} is a pair $(r,\vvvv) \in  \mathcal C ^{\infty}( \coprod_{i,j \in I}U_{ij},\HH ) \times \mathcal C ^{\infty} (
\coprod_{i,j,k \in J}U_{ijk},\G) $. We say that a $\G \to \HH$-valued $1$-coboundary  $ (r,\vvvv) $,  \emph{relates} two non-Abelian $1$-cocycles  $(\lambda,\gggg) $ and $(\lambda',\gggg') $ if,
\begin{equation}\label{eq:relate}   \left\{ \begin{array}{rcll}
                   \lambda_{ij}'      &=&  \rho(\vvvv_{ij})   r_i
                    \lambda_{ij}  r_j^{-1}, & (*) \\
\gggg_{ijk}'  \vvvv_{ik} &=& \lambda_{ij}' (\vvvv_{jk}) \vvvv_{ij}
                   r_i(\gggg_{ijk}),  & (**) \\ \end{array}
\right. \end{equation}
for all possible indices. We  recall that $r_i, \vvvv_{ij}$ stand for the restriction of non-Abelian $1$-coboundary  $(r,\vvvv)$ to the intersection $U_{ij}$.
\end{defi}

The next proposition relates coboundaries and isomorphisms of adapted extensions which generalizes the results of \cite{BL} to arbitrary crossed modules.

\begin{prop}\label{prop:coboundaries}  Let $(\mathcal U, \bullet,\star )$ and $(\mathcal U, \bullet',\star')$ be two adapted Lie groupoid $\G \to \HH$-extensions of $\Cechs$.
Let $(\lambda,\gggg) $ and $(\lambda',\gggg')$ be the $\G \to \HH$-valued  non-Abelian $1$-cocycles w.r.t. ${\mathcal U}$
  associated with the adapted Lie groupoid $\G\to\HH$-extensions $(\mathcal U, \bullet,\star )$ and $(\mathcal U, \bullet',\star')$, respectively (as in proposition \ref{prop:cocycles}).
Then, the following construction defines a one-to-one correspondence between
 the set of isomorphisms of Lie groupoid $\G \to \HH$-extensions of $\Cechs$
from $(\mathcal U, \bullet,\star )$
to $(\mathcal U, \bullet',\star' )$,
and the set of $\G\to\HH$-valued $1$-coboundaries relating $(\lambda,\gggg) $ and $(\lambda',\gggg')$:
\begin{enumerate}
\item
Given an isomorphism $(\Phi_{\mathcal R}, \Phi_P) $ of (adapted) Lie groupoid $\G \to \HH$-extensions of $\Cechs$ between $(\mathcal U, \bullet,\star )$ and $(\mathcal U, \bullet',\star')$, we define $r_i: U_i \to \HH$ and $\vvvv_{ij}: U_{ij} \to \G$ by:
  \begin{equation}\label{eq:isomclasses} \begin{array}{rcl} ( x_{i},r_i) &=&  \Phi_P ( x_{i},e),
  \\ (\vvvv_{ij}^{-1}, x_{ij})  &=& \Phi_{\mathcal R} (e, x_{ij}).  \end{array} \end{equation}

\item Given a $\G \to \HH$-valued $1$-coboundary $(r , \vvvv) $ such that relates the non-Abelian $1$-cocycles
$(\lambda,\gggg) $ and $(\lambda',\gggg') $, define an isomorphism of Lie groupoid $\G\to \HH$-extensions $(\Phi_{\mathcal R}, \Phi_P) $ between the corresponding adapted Lie groupoid $\G\to \HH$-extensions $(\mathcal U, \bullet,\star )$ and $(\mathcal U, \bullet',\star')$ as follows:
  \begin{equation} \label{eq:def_PhiX}
  \Phi_\mathcal R (g,x_{ij}) =   (r_i (g) \vvvv_{ij}^{-1}, x_{ij} ), \quad \quad \mbox{for all}\quad i,j \in I, x \in U_{ij}, g \in \G
  \end{equation}
  and
    \begin{equation} \label{eq:def_PhiP}
    \Phi_P (x_{i},h) = (x_i ,  r_i h),\quad \quad \mbox{for all}\quad i,j \in I, x \in U_{ij}, g \in \G.
    \end{equation}
  \end{enumerate}
\end{prop}
\begin{proof}
1) First we prove that given an isomorphism of Lie groupoid $\G \to \HH$-extensions $(\Phi_{\mathcal R}, \Phi_P)$ between the adapted Lie groupoid $\G \to \HH$-extensions $(\mathcal U, \bullet,\star )$
and $(\mathcal U, \bullet',\star' )$, by following the construction in item 1 we obtain a $\G \to \HH$-valued $1$-coboundary. For this we need
 to prove that the pair $(r , \vvvv)$ obtained as in (\ref{eq:isomclasses})
satisfy relations (\ref{eq:relate}). We first prove the first of those relations, by
exploiting the fact that $(\Phi_{\mathcal R},\Phi_P, id_{\HH})$ is a morphism of principal bundles over Lie groupoids (see Definition \ref{defi:isomppalbundles}), which amounts to:
\begin{equation}\label{eq:cobound1}  \Phi_P (  (e,x_{ij}) \star (x_j,e) )  =
\Phi_\mathcal R (  (e,x_{ij}))  \star' \Phi_P ( (x_j,e) ), \quad \quad \forall ij \in I , x \in U_{ij}.
\end{equation}
The LHS of (\ref{eq:cobound1}) is given by
\begin{equation*}
\begin{array}{rcll}
\Phi_P (  (e,x_{ij}) \star (x_j,e) ) & = & \Phi_P (x_{i},\lambda_{ij})  & \hbox{by (\ref{eq:cons}), i.e. definition of $\lambda_{ij}$}\\
                                              & = & \Phi_P (x_{i},e) \cdot \lambda_{ij} & \hbox{$\Phi_P$ being a $\HH$-ppal bundle morphism}\\
                                              & = & (x_i,r_i) \cdot \lambda_{ij} & \hbox{by (\ref{eq:isomclasses}), i.e. definition of $r_i$}\\
                                              & = & ( x_i , r_i\lambda_{ij}) & \hbox{by definition \ref{definition of Adapted extension}, item 2.} \\
 \end{array}
 \end{equation*}
While the RHS of (\ref{eq:cobound1}) is given by
\begin{equation*}
\begin{array}{rcll}
                             & & \Phi_\mathcal R   (e,x_{ij})  \star' \Phi_P  (x_j,e) & \\
                            & = & (\vvvv_{ij}^{-1},x_{ij}) \star' (x_j, r_j) & \hbox{by  (\ref{eq:isomclasses}), i.e. def. of $\vvvv_{ij}^{-1}$  and $r_j$ }\\
                             & = & (( \vvvv_{ij}^{-1}, x_{ii}) \bullet'  (e,x_{ij})) \star' (x_j, e)\cdot r_j &  \hbox{by def. \ref{definition of Adapted extension}  item 2 and 4} \\
                              & = &( \vvvv_{ij}^{-1}, x_{ii}) \star' ( x_i, \lambda'_{ij}) \cdot r_j & \hbox{by  (\ref{eq:cons}), i.e. def. of $\lambda_{ij}'$}\\
                              & = & ( \vvvv_{ij}^{-1}, x_{ii}) \star' ( x_i, \lambda'_{ij} r_j) & \hbox{by def. \ref{definition of Adapted extension}  item 2}\\
                             & = & ( x_i, \rho(\vvvv_{ij}^{-1})\lambda'_{ij} r_j ) & \hbox{by (\ref{defi: ppalstronP})}
\end{array}
\end{equation*}

Comparing the LHS and RHS of (\ref{eq:cobound1}), we get
$$r_i\lambda_{ij} = \rho(\vvvv_{ij}^{-1})\lambda'_{ij} r_j$$ or $$\lambda'_{ij}= \rho(\vvvv_{ij})r_i\lambda_{ij}  r_j^{-1},$$
 which is the first relation of (\ref{eq:relate}).
Before proving the second relation of (\ref{eq:relate}), we need to explore the consequences of the commutativity of the diagram displayed in (\ref{diagram for isomorphism}).
It follows from item 3 in definition \ref{definition of Adapted extension} that $ \chi((x_i,e))$ is the element in the band given by $\chi((x_i,e))(g)=  ( g, x_{ii} ) $, so that $ \overline{\Phi_\mathcal R}(\chi((x_i,e))) $ is by definition the element of the band given by $g \mapsto \Phi_\mathcal R((g,x_{ii})) $.
Now, $\Phi_P((x_i,e))= (r_i,e) $ by (\ref{eq:isomclasses}), i.e. definition of $r_i$, so that $\chi'( \Phi_P((x_i,e)) ) $ is the element of the band given by
$g \mapsto (r_i(g),x_{ii}) $, by item (3) of definition of adapted extensions again. The commutativity of diagram (\ref{diagram for isomorphism}) can therefore be expressed by meaning that the next relation holds for all $g \in \G$:
 \begin{equation}\label{eq:PhiX} \Phi_\mathcal R(g,x_{ii}) = (r_i(g),x_{ii}) .\end{equation}
 Exploiting the assumption that $\Phi_\mathcal R$ is a Lie groupoid morphism, we can derive a more general formula as follows
\begin{equation}  \label{eq:cobound2}
\begin{array}{rcll}
\Phi_\mathcal R(g,x_{ij})&=&\Phi_{\mathcal R}((g,x_{ii})\bullet (e,x_{ij}))& \hbox{by definition \ref{definition of Adapted extension} item 4}\\
   &=&\Phi_\mathcal R(g , x_{ii} )\bullet'\Phi_\mathcal R(e,x_{ij})& \hbox{$\Phi_{\mathcal R}$ being a Lie groupoid morphism}\\
                &=&(r_i(g) , x_{ii} )\bullet'\Phi_\mathcal R(e,x_{ij})& \hbox{by (\ref{eq:PhiX})}\\
                &=&(r_i(g) , x_{ii} )\bullet'(\vvvv_{ij}^{-1}, x_{ij})& \hbox{by (\ref{eq:isomclasses}) definition of $ \vvvv_{ij}$}\\
                &=&(r_i(g)\vvvv_{ij}^{-1}, x_{ij})& \hbox{by definition \ref{definition of Adapted extension} item 4.}
\end{array}
\end{equation}
Now, we derive the second of the relations (\ref{eq:relate}) by comparing the left and right hand sides of a relation following from the assumption that $\Phi_\mathcal R$ be a Lie groupoid morphism:
 \begin{equation}  \label{eq:cobound3}
 \Phi_\mathcal R( (e,x_{ij}) \bullet (e,x_{jk}) ) = \Phi_{\mathcal R}  ((e,x_{ij})) \bullet' \Phi_\mathcal R ((e,x_{jk})),
 \end{equation}
a computation that goes as follows:
\begin{equation*}
\begin{array}{rcll}
                            & & \Phi_\mathcal R ( (e,x_{ij}) \bullet (e,x_{jk}) )& \\
                            & = & \Phi_\mathcal R ( \gggg_{ijk} ,x_{ik} )  & \hbox{by (\ref{eq:cons}), i.e. definition of $\gggg_{ijk}$ }\\
                            & = &    (r_i(\gggg_{ijk})\vvvv_{ik}^{-1},x_{ik}) & \hbox{by (\ref{eq:cobound2})},  \\
\end{array}
\end{equation*}
while the RHS of (\ref{eq:cobound3}) is:
\begin{equation*}
\begin{array}{rcll}
 & & \Phi_\mathcal R  ((e,x_{ij})) \bullet' \Phi_\mathcal R ((e,x_{jk}))& \\
 & =&  ( \vvvv_{ij}^{-1}, x_{ij} ) \bullet' ( \vvvv_{jk}^{-1}, x_{jk} )  & \hbox{by (\ref{eq:isomclasses}), i.e. def. of $\vvvv_{ij}$ and $\vvvv_{jk} $ } \\
  & = & (  \vvvv_{ij}^{-1}  \lambda'_{ij}(\vvvv_{jk}^{-1}) \gggg_{ijk}', x_{ik})  & \hbox{by (\ref{eq:recons})}. \\
\end{array}
\end{equation*}
Comparing the LHS  and the RHS of (\ref{eq:cobound3}) leads to
$$ r_i(\gggg_{ijk})\vvvv_{ik}^{-1} = \vvvv_{ij}^{-1}  \lambda'_{ij}(\vvvv_{jk}^{-1}) \gggg_{ijk}'
\hbox{ $\Leftrightarrow $ } \gggg_{ijk}'\vvvv_{ik}=\lambda'_{ij}(\vvvv_{jk})\vvvv_{ij}r_i(\gggg_{ijk}) ,$$
which is precisely the second relation of (\ref{eq:relate}), and completes the proof of the first item.

\bigskip
\noindent
2) Second, we prove that given a $\G \to \HH$-valued $1$-coboundary, by following the construction in item 2, we get an isomorphism of adapted Lie groupoid $\G \to \HH$-extensions. In order to show that the triple $(\Phi_\mathcal R, \Phi_P, \id_{\HH}) $ with $\Phi_\mathcal R, \Phi_P$ as in (\ref{eq:def_PhiX}) and (\ref{eq:def_PhiP}), is an isomorphism of Lie groupoid $\G \to \HH$-extensions, we need to check that (see Notation \ref{not:adaptextensions})
 \begin{enumerate}
 \item[(a)] $\Phi_\mathcal R : \mathcal R \to \mathcal R'$ is a morphism of Lie groupoids,
 \item[(b)] $\Phi_P : P \to P'$ is a morphism of principal bundles over Lie groupoids,
 \item[(c)] the following diagram commutes:
 $$ \xymatrix{P \ar [r]^{\Phi_P} \ar [d]^{\chi}& P'\ar [d]_{\chi'}\\
              Band((\mathcal R,\bullet)\to \Cechs)\ar [r]^{\bar{\Phi}_\mathcal R} & Band((\mathcal R,\bullet')\to \Cechs)} $$
              with $\bar{\Phi}_\mathcal R$ being defined as in (\ref{diagram for isomorphism})
 \end{enumerate}
We first check that condition (a) holds, i.e that $\Phi_\mathcal R( r \bullet r'  ) = \Phi_\mathcal R(r) \bullet' \Phi_\mathcal R(r') $
for arbitrary elements of the form $r=(g,x_{ij}) \in \mathcal R$ and $r'=(g',x_{jk}) \in \mathcal R$. On the one hand:
\begin{equation*}
\begin{array}{rcll}
\Phi_\mathcal R((g,x_{ij}) \bullet (g',x_{jk})) & = & \Phi_\mathcal R (g \lambda_{ij} ( g')  \gggg_{ijk}, x_{ik}) & \hbox{by (\ref{eq:recons}) in prop. \ref{prop:cocycles}}                      \\
 & = & ( r_i ( g \lambda_{ij} (g') \gggg_{ijk}) \vvvv_{ik}^{-1}, x_{ik}) & \hbox{by (\ref{eq:def_PhiX}), i.e. definition of $\Phi_\mathcal R$},
\end{array}
\end{equation*}
while on the other hand:
\begin{equation*}
\begin{array}{rcll}
  & & \Phi_\mathcal R( (g,x_{ij}) ) \bullet' \Phi_\mathcal R (g',x_{jk}) & \\
  & = &  (r_i(g)\vvvv_{ij}^{-1},x_{ij})   \bullet'  (r_j(g')\vvvv_{jk}^{-1},x_{jk}) & \hbox{by (\ref{eq:def_PhiX}), i.e. def. of $\Phi_\mathcal R$} \\
  & = & (   r_i(g)\vvvv_{ij}^{-1}  \lambda_{ij}'( r_j(g')\vvvv_{jk}^{-1}) \gggg_{ijk}' ,x_{ik}) & \hbox{by (\ref{eq:recons}) in prop.  \ref{prop:cocycles}.}             \\
\end{array}
\end{equation*}
Of course, $\Phi_\mathcal R$ is a Lie groupoid isomorphism if and only if both sides of the previous relations are equal for all $g,g' \in \G$, i.e. if and only if
\begin{equation}\label{eq:bothsides} r_i ( g \lambda_{ij} (g') \gggg_{ijk}) \vvvv_{ik}^{-1}  = r_i(g)\vvvv_{ij}^{-1}  \lambda_{ij}'( r_j(g')\vvvv_{jk}^{-1}) \gggg_{ijk}'\end{equation}
which reduces, multiplying both sides by $r_i(g^{-1})$, to require that, for all $g' \in \G $:
$$ r_i ( \lambda_{ij} (g') \gggg_{ijk}) \vvvv_{ik}^{-1}  =\vvvv_{ij}^{-1}  \lambda_{ij}'( r_j(g')\vvvv_{jk}^{-1}) \gggg_{ijk}'$$
or, equivalently:
$$r_i ( \lambda_{ij} (g') ) r_i( \gggg_{ijk} ) \vvvv_{ik}^{-1} =  \vvvv_{ij}^{-1}  \lambda_{ij}'( r_j(g')) \vvvv_{ij} \vvvv_{ij}^{-1} \lambda_{ij}'(\vvvv_{jk}^{-1})   \gggg_{ijk}' $$
 By the second relation in (\ref{eq:relate}), $ r_i( \gggg_{ijk} ) \vvvv_{ik}^{-1} =  \vvvv_{ij}^{-1} \lambda_{ij}'(\vvvv_{jk}^{-1})   \gggg_{ijk}' $,
 so that, eventually,  $\Phi_\mathcal R$ is a Lie groupoid isomorphism if and only if for all $g ' \in \G$
\begin{equation}\label{eq:*}
  r_i ( \lambda_{ij} (g') )= \vvvv_{ij}^{-1}  \lambda_{ij}'( r_j(g')) \vvvv_{ij}
\end{equation}
By axiom of crossed module the RHS of (\ref{eq:*}) is equal to
$\rho(\vvvv_{ij}^{-1})\lambda_{ij}'(r_j (g'))$
so that eventually $\Phi_\mathcal R$ is a Lie groupoid isomorphism if and only if
$$r_i \circ \lambda_{ij} \,(g')= \rho(\vvvv_{ij}^{-1}) \circ \lambda_{ij}' \circ r_j\, (g'); \forall g' \in \G ,$$
an equation which is obtained by applying $\jmath: \HH \to \Aut (\G)$ to the first relation in (\ref{eq:relate}), and is therefore true, here $\circ$ refers to the composition low of $\Aut(\G)$,
Hence, $\Phi_\mathcal R $ is a Lie groupoid isomorphism.

We wish now to check that condition (b) holds, i.e that $\Phi_P( r \star p  ) = \Phi_\mathcal R(r) \star' \Phi_P(p) $
for arbitrary elements $r=(g,x_{ij}) \in \mathcal R$ and $p=(x_i,h) \in P$. On the one hand, we compute:
\begin{equation} \label{eq:condition4}
\begin{array}{rcll}
                                          & & \Phi_P((g,x_{ij}) \star (x_{j},h)) & \\
                                          & = & \Phi_P(x_i,\rho(g)\lambda_{ij}h)  & \hbox{by (\ref{eq:recons}) in prop. \ref{prop:cocycles}}\\
                                          & = & (x_i, r_i \rho(g)\lambda_{ij}h )  & \hbox{by (\ref{eq:def_PhiP}), i.e. definition of $\Phi_P$},
\end{array}
\end{equation}
while on the other hand, we compute:
\begin{equation} \label{eq:condition5}
\begin{array}{rcll}
                       &   & \Phi_\mathcal R(g,x_{ij}) \star'\Phi_P(x_{j},h)) &\\
                       & = & (r_i(g)\vvvv^{-1}_{ij},x_{ij}) \star' (x_j, r_jh) & \hbox{by (\ref{eq:def_PhiX}-\ref{eq:def_PhiP}), i.e. def. of $\Phi_\mathcal R $ and  $\Phi_P$}\\
                       & = & (x_i, \rho(r_i(g)\vvvv^{-1}_{ij})\lambda'_{ij}r_j h) & \hbox{by (\ref{eq:recons}) in prop. \ref{prop:cocycles}}
\end{array}
\end{equation}
Equations (\ref{eq:condition4}) and (\ref{eq:condition4}),  together with $\rho(r_i(g)\vvvv^{-1}_{ij})\lambda'_{ij}r_j h= r_i \rho(g)\lambda_{ij}h$
(an immediate consequence of (\ref{eq:relate})), imply that:
$$ \Phi_P((g,x_{ij}) \star(h,x_{j}))= \Phi_\mathcal R(g,x_{ij}) \star'\Phi_P(h,x_{j}))$$
which completes the proof of (b). Condition (c) is a direct computation.

Last, we have to check that both constructions in item 1 and item 2 are inverse one to the other. It is easy to see that, applying the construction of item 2 and  then the construction of item 1 to a $\G \to \HH $-valued $1$-coboundary $(r_i,\vvvv_{ij}) $, one obtains $(r_i,\vvvv_{ij}) $.
Moreover,  two
  $(\Phi_\mathcal R,\Phi_P) $, $(\Phi_\mathcal R',\Phi_P') $
  isomorphisms of $\G \to \HH$-extensions  which correspond to the same coboundary $(r_i,\vvvv_{ij}) $ need to be equal. This follows from (\ref{eq:cobound2}), which clearly implies that $\Phi_{\mathcal R}=\Phi_\mathcal R'$, and from (\ref{eq:isomclasses}), which implies that $\Phi_P$ and $\Phi_{P'}$ coincide on every element in $P$ of the form $(x_{i},e) $, and are therefore equal since principal bundle morphisms that coincide on some global section coincide globally. This completes the proof.
\end{proof}

Let $ {\mathcal U}=(U_i)_{i \in I}$ be an open covering of a manifold $N$,
and $\G \to \HH$ be a crossed module of Lie groups.
It follows from proposition \ref{prop:coboundaries} that coboundaries define an equivalence relation on
the set of $\G \to \HH$-valued $1$-cocycles w.r.t. the open covering ${\mathcal U} $. The quotient set obtained by this equivalent relation is called \emph{$\G \to \HH$-valued $1$-cohomology w.r.t. the open covering ${\mathcal U}$} and is denote by $H^1_{\mathcal U}( \G \to \HH)$ .

The next corollary follows from propositions \ref{prop:cocycles} and
 \ref{prop:coboundaries}.

\begin{cor}\label{cor:1-1Cohomology-adapted}
Let $ {\mathcal U}=(U_i)_{i \in I}$ be an open covering of a manifold $N$,
and $\G \to \HH$ a crossed module of Lie groups. There is a one-to-one correspondence between the set $H^1_{\mathcal U}( \G \to \HH)$
and the set of all adapted Lie groupoid $\G\to \HH$-extensions of $\Cechs $ up to isomorphisms (of Lie groupoids $\G \to \HH $-extensions of $\Cechs$).
\end{cor}

The notion of adapted extension may appear to be somewhat arbitrary. We wish to convince the reader that it is not,
 by showing the next proposition.

\begin{prop}\label{prop:everyextensionisadapted}
Let $ {\mathcal U}=(U_i)_{i \in I}$ be an open covering of a manifold $N$ such that $U_{ij}$
is a contractible open set for all $i,j \in I$,
and let $\G \to \HH$ be a crossed module of Lie groups.
Then every Lie groupoid $\G \to \HH$-extension of the \v{C}ech groupoid $\Cechs $ is isomorphic (as a Lie groupoid $\G \to \HH$-extension of $\Cechs$) to an adapted Lie groupoid $\G \to \HH$-extension $\Cechs$.
\end{prop}
\begin{proof}
Let $( \mathcal R \stackrel{\phi}{\to} \Cechs, P\to \coprod_{i \in I}U_i, \chi) $ be a Lie groupoid $\G \to \HH$-extension of $\Cechs$.
Since $\coprod_{i \in I}U_i $ is a disjoint union of contractible sets (since $U_{ij}$ is by assumption contractible for all $i,j \in I$, so is $U_i = U_{ii}$), there exists a global section $\sigma $ of the principal $\HH$-bundle $ P\to \coprod_{i \in I}U_i$.

Since $\chi: P \to Band (\mathcal R \stackrel{\phi}{\to} \Cechs)$ is by assumption a morphism of principal bundles over the identity of $ \coprod_{i \in I}U_i $, the map $\hat{\sigma} :=\chi \circ \sigma  $ is a global section of the principal $Aut(\G)$-bundle $Band( \mathcal R \stackrel{\phi}{\to} \Cechs )$.
In turn, a global section of the band amounts to a global trivialization of the kernel $K \to  \coprod_{i \in I}U_i$, by considering the group bundle isomorphism $ \tau_{K} :\G \times \coprod_{i \in I} U_{i}  \simeq K $ given by
 $ (g,x_i) \mapsto  \hat{\sigma}(x_i) (g) $.
 Since, by construction, $\hat{\sigma}(x_i) $ belongs to $Band_{x_i}= Aut(\G,K_{x_i}) $, it is clear that
 $\tau_K $ is, as expected, a group bundle isomorphism over the identity of $\coprod_{i \in I} U_{ii} $.

 Now, the surjective submersion $\phi: \mathcal R \to  \coprod_{i,j \in I} U_{ij} $ restricts to a surjective submersion
 from $\mathcal R  \backslash K $ to $\coprod_{i \neq j} U_{ij} $, and the fibers of this submersion are acted upon transitively and freely by $K $. Using $\tau_K $, we endow
 $$\mathcal R  \backslash K  \to \coprod_{i,j \in I \hbox{ s.t.} i \neq j} U_{ij}$$
with a structure of principal $\G$-bundle as follows: the outcome of the action of  $g \in \G $ on $r \in  \mathcal R  \backslash K $
is defined to be $ \tau_K (g, {s(r)} ) \bullet_{\mathcal R} r   $.
Every principal bundle over a disjoint union of contractible open sets is trivial, which means, in this case, that there is a global section
$\sigma_{1}:\coprod_{i \neq j} U_{ij} \to \mathcal R\backslash K $.
Then we define $\tau_{\mathcal R \backslash K } : \G \times \coprod_{i \neq j} U_{ij} \to \mathcal R \backslash K $  by
$$ (g,x_{ij}) \mapsto \tau_K(g,x_i) \bullet_{\mathcal R} \sigma_1(x_{ij}),$$
 for all $i,j \in I, i \neq j$.
By construction, $\tau_{\mathcal R \backslash K } $ is a group bundle morphism over the identity of $\coprod_{i \neq j }U_{ij} $.
Gluing $\tau_{ K}  $ and $ \tau_{\mathcal R \backslash K}$, we get a map (over the identity of $\coprod_{i,j \in I}U_{ij} $)
that we denote by $\tau : \G \times \coprod_{i,j \in I} U_{ij} \to \mathcal R $, namely:
 $$ \tau(g,x_{ii}):= \tau_K(g,x_i),\quad \forall i\in I $$ and $$ \tau(g,x_{ij}):= \tau_{\mathcal R  \backslash K }(  g,x_{ij} ),\quad \forall i,j \in I  \hbox{with} i \neq j.$$
 The section $\sigma$ of $ P \to  \coprod_{i \in I} U_i$ also, induces a map
   $\Psi_P :   \coprod_{i \in I} U_{i} \times \HH \simeq  P $ given by:
   \begin{equation} \label{eq:psiP}
    (x_{i}, h)  \mapsto \sigma (x_i) \cdot  h.
    \end{equation}
With the help of this pair of maps $\Psi_P $ and $\tau$, the structure of $\G \to \HH$-extensions on $ (\mathcal R \stackrel{\phi}{\to} \Cechs, P\to \coprod_{i \in I}U_i, \chi)$ is transported and induces a structure of $\G \to \HH$-extension on $( \G \times \coprod_{i,j \in I} U_{ij} \stackrel{\phi}{\to} \Cechs,  \coprod_{i} U_{i} \times \HH \to \coprod_{i \in I}U_i, \chi' )  $. Explicitly the induced Lie groupoid structure on $\G\times \coprod_{i,j\in I}U_{ij} \toto \coprod_{i\in I}U_{i}$ is given by :
$$(g,x_{ij})\bullet(g',x_{jk}):=\tau^{-1}(\tau(g,x_{ij})\bullet_{\mathcal R} \tau(g',x_{jk})),$$
for all $g,g'\in\G, i,j,k\in I, x\in U_{ijk}$, the induced action of Lie groupoid $G\times \coprod_{i,j\in I}U_{ij} \toto \coprod_{i\in I}U_{i}$ on $\coprod_{i\in I}U_{i}\times \HH$ is given by:
$$(g,x_{ij})\star(x_j,h):=\Psi_P^{-1}(\tau(g,x_{ij}) \bullet_{\mathcal R,P}\Psi_P(x_j,h)),$$
for all $g\in\G,h\in\HH i,j\in I, x\in U_{ij}$ and the induced principal bundle structure on $\coprod_{i\in I}U_{i}\times \HH \to \coprod_{i}U_{i}$ over the Lie groupoid $\G\times \coprod_{i,j\in I}U_{ij} \toto \coprod_{i,j\in I}U_{ij}$ is given by:
$$(x_i,h)h'=(x_i,hh'),$$
for all $h,h'\in\HH i\in I, x\in U_{i}$.
Last we define $ \chi':\coprod_{i\in I}U_{i} \times \HH \to Band(\G\times \coprod_{i,j\in I}U_{ij} \toto \coprod_{i,j\in I}U_{ij})$ by
$$ (x_i,h)\mapsto (x_i,j(h)),$$
for all $h\in\HH i\in I, x\in U_{i}$.
We claim that:
\begin{enumerate}
\item the extension $Ext_2:=(G\times \coprod_{i,j \in I}U_{ij} \to \coprod_{i,j \in I}U_{ij}, \coprod_{i\in I}U_{i} \times \HH \to \coprod_{i \in I}U_{i},\chi' )$ is a Lie groupoid $\G \to \HH$-extension.
\item the Lie groupoid $\G \to \HH$-extension $Ext_2$ is isomorphic to the  Lie groupoid $\G \to \HH$-extension $Ext_1:=(\mathcal R \to \coprod_{i,j \in I}U_{ij}, P \to \coprod_{i \in I}U_{i}, \chi)$.
\item the Lie groupoid $\G \to \HH$-extension $Ext_2$ is an adapted Lie groupoid $\G \to \HH$-extension.
\end{enumerate}
These claims complete the proof of the proposition.
For the proof of claim 1), it is enough to check that $(x_i,h)\cdot \rho(g)=\chi'(x_i,h)(g)\star(x_i,h)$ for all $x\in N,i\in I,h\in \HH, g\in \G,$ which goes as follows:
\begin{equation}
\begin{array}{rcll}
                        & &\chi'(x_i,h)(g)\star (x_i,h) & \\
                        &=& (h(g),x_{ii}) \star (x_i,h)& \hbox{by def of $\chi'$}\\
                        &=& \Psi_P^{-1}(\tau (h(g),x_{ii})\bullet_{\mathcal R,P}\Psi_P(x_i,h)) & \hbox{by def of $\star$}\\
                        &=& \Psi_P^{-1}(\chi \circ \sigma (x_i)(h(g)) \bullet_{\mathcal R,P}\sigma(x_i)\cdot h) &\hbox{by def of $\tau$ and def of $\Psi_P$}\\
                        &=& \Psi_P^{-1}( \chi(\sigma(x_i)\cdot h)(g)\bullet_{\mathcal R,P}\sigma(x_i)\cdot h)& \hbox{$\chi$ is morphism of ppal bundles}\\
                        &=& \Psi_P^{-1}(\sigma(x_i)\cdot h \cdot \rho(g))& \hbox{since $Ext_1$ is a $\G \to \HH$-extension}\\
                        &=& (x_i,h\cdot\rho(g))& \hbox{by def of $\Psi_P^{-1}$}\\
                        &=& (x_i,h)\cdot\rho(g).&

\end{array}
\end{equation}
For the proof of claim 2), since  $\Psi_P$ is a morphisms of principal bundles and $\tau$ is a morphism of Lie groupoids, it is enough to prove that the following diagram commutes: \label{diag*}
$$\xymatrix{\coprod_{i\in I}U_i \times\HH \ar[d]^{\chi'} \ar[rrr]^{\Psi_P}&& & P \ar[d]\\
             Band(\G\times\coprod_{i,j \in I}U_{ij} \ar[r]& \coprod_{i,j \in I}U_{ij})\ar[r]^{\bar{\tau}} &Band(\mathcal R \ar[r] &\coprod_{i,j\in I}U_{ij} ) }$$
In turn, the commutativity of this diagram is proved by the following computations:
\begin{equation*}
\begin{array}{rcll}
                     & & (\bar{\tau} \circ \chi' (x_i,h))(g) & \\
                     &=& \tau(\chi'(x_i,h)(g)) & \hbox{by def of $\bar{\tau}$}\\
                     &=& \tau(h(g),x_{ii}) & \hbox{by def of $\chi'$}\\
                     &=& (\chi\circ \sigma (x_i))(h(g)) & \hbox{by def of $\tau$}\\
                     &=& \chi \circ \sigma(x_i) \circ j(h)(g) &\hbox{ }\\
                     &=& \chi(\sigma(x_i)\cdot h)(g)& \hbox{since $\chi$ is a morphism of ppal bundles}\\
                     &=& \chi \circ \Psi_P(x_i, h)(g)& \hbox{by def of $\Psi_P $},

\end{array}
\end{equation*}
for all $i\in I , x_i \in U_i, h \in \HH, g\in \G$. \\
 Last we prove the claim 3). For this, it is enough to check that the axiom 4 in definition \ref{definition of Adapted extension} holds, while the other axioms in definition \ref{definition of Adapted extension} hold by construction. We show that
 \begin{equation}\label{eq:adap}
 (g,x_{ii})\bullet (g',x_{ij})=(gg',x_{ij}),
 \end{equation}
  for all $x \in N, i,j \in I , g,g' \in \G$. This goes as follows:
\begin{equation*}
\begin{array}{rcll}
                     & & \hbox{LHS of (\ref{eq:adap})}& \\
                     &=& \tau^{-1}(\tau(g,x_{ii})\bullet_{\mathcal R}\tau(g',x_{ij}))&\hbox{by def of $\bullet$}\\
                     &=& \tau^{-1}(\chi \circ \sigma(x_i)(g)\bullet_{\mathcal R}\chi \circ\sigma(x_i)(g')\bullet_{\mathcal R}\sigma_1(x_{ij}))&\hbox{by def of $\tau$}\\
                     &=& \tau^{-1}(\chi \circ\sigma(x_i)(gg')\bullet_{\mathcal R} \sigma_1(x_{ij}))&\hbox{$\chi \circ\sigma(x_i)$ is a morphism of groups}\\
                     &=& \tau^{-1}(\tau(gg',x_{ij}))&\hbox{by def of $\tau$ }\\
                     &=& \hbox{RHS of (\ref{eq:adap})}.&
\end{array}
\end{equation*}

\end{proof}

We can now state the conclusion of this section, which follows immediately from proposition \ref{prop:everyextensionisadapted} and corollary \ref{cor:1-1Cohomology-adapted}.

\begin{them}
Let $ {\mathcal U}=(U_i)_{i \in I}$ be an open covering of a manifold $N$ such that $U_{ij}$
is a contractible open set for all $i,j \in I$ and $\G \to \HH$ be a crossed module of Lie groups.
Then, there is a one-to-one correspondence between
\begin{enumerate}
\item[(i)] the set $H^1_{\mathcal U}( \G \to \HH)$,
\item[(ii)] adapted Lie groupoid $\G \to \HH$-extensions of the \v{C}ech groupoid $\Cechs $, up to isomorphisms of Lie groupoid $\G \to \HH$-extensions of $\Cechs $,
\item[(iii)] Lie groupoid $\G \to \HH$-extensions of $\Cechs $ up to isomorphisms of Lie groupoid $\G \to \HH$-extensions of $\Cechs $.
\end{enumerate}
\end{them}
\begin{proof}
The equivalence between (i) and (ii) was already stated in corollary \ref{cor:1-1Cohomology-adapted}.
The equivalence between (iii) and (ii) comes from proposition \ref{prop:everyextensionisadapted} which states that every Lie groupoid $\G \to \HH$-extension
 of $\Cechs $ is isomorphic to an adapted one, of course, a given extension can be isomorphic (as Lie groupoid $\G \to \HH$-extensions of $\Cechs $)
 to two different adapted $\G \to \HH$-extensions, but both adapted extensions are then isomorphic (as Lie groupoid $\G \to \HH$-extensions of $\Cechs $), so that the assignment from (iii) to (ii) is well-defined. This concludes the proof of the theorem.
\end{proof}

\section{Morita equivalence of Lie groupoid $\G \to \HH$-extensions, and $\G \to \HH$-gerbes on groupoids.}\label{sec:MoritaExt}

Let $\G \to \HH$ be a crossed module. We intend in this section to define, purely in terms of Lie groupoids, the notion of $\G \to \HH $-gerbes over a given Lie groupoid $\BBB \toto \BBB_{0}$, having in mind the case where $\BBB \toto \BBB_{0}$ is the trivial Lie groupoid $N \toto N$ associated to a manifold~$N$.
In view of the preceding section, it is reasonable to consider all the $\G \to \HH$-extensions of all the possible pull-back of $\BBB \toto \BBB_{0}$ with respect to surjective submersions. For instance, when $\BBB \toto \BBB_{0} $ is of the form $N \toto N$, with $N$ a manifold, this includes all the $\G \to \HH$-extensions of the \v{C}ech groupoids associated to an arbitrary open cover of $N$ (because the \v{C}ech groupoid $N[{\mathcal U}]  \toto \coprod_{i \in I}U_i$ is the pull-back groupoid of $N \toto N $ with respect to the natural inclusion maps $\imath: \coprod_{i \in I}U_i \to N $).
  But of course, we shall later have to take a quotient of that class, which is way too large. We do it by identifying
two $\G \to \HH$-extensions which are Morita equivalent in some sense described below.

Following several comments from the referee, we would like to say a few words about the link between the present work and an article by Ginot and Sti\'enon \cite{GinotStienon}.
As mentioned in Remark \ref{rem:linkGinotStienon}, there is a clear relation between $ \G \to \HH$-extensions and $ \G \to \HH$-bundles in their sense, and we have no doubt that we could have completed the purpose of this section by using their language. There were two reasons not to do so. First, we wanted not to use Lie $2$-groupoids, a self-imposed limitation that can criticized, of course, but we feel that Lie $2$-groupoids would be a too demanding notion for some mathematicians willing to study non-Abelian gerbes.  Second, one of our concern was to address the problem of knowing when two extensions should be identified. We do not want to identify, in the case of the crossed-module $ 1  \to \HH $, case for which extensions are simply $\HH$-principal bundles (see example \ref{ex:1toHH}), a principal bundle and its pull-back through a diffeomorphism of the base. To overcome this difficulty, we introduced the notion of
$ \G \to \HH$-extensions over a given Lie groupoid that shall appear below together with the appropriate Morita equivalences.
Of course, we have no doubt that the point of view of  \cite{GinotStienon} could also be adapted in order to make these identifications precise, but this does not appear to us to be that trivial. However, we acknowledge that the point of view of $ \G \to \HH$-principal bundles is also an efficient tool, and, in a subsequent work, we hope to make the relation more precise.

\subsection{Definition of Morita equivalence of $ \G \to \HH$-extensions and $\G \to \HH $-gerbes}

 Let us first define what the pull-back of a $ \G \to \HH$-extension is.

 Given a Lie groupoid extension $\mathcal{R} \stackrel{\phi}{\to} \GGG \toto M$ and a surjective submersion $p:M' \to M$, the functor of definition \ref{def:pback} applied to $\mathcal{R} \stackrel{\phi}{\to} \GGG $ yields a Lie groupoid extension
  $$\mathcal{R}[p] \stackrel{\phi[p]}{\longrightarrow} \GGG[p].$$
  It is routine to check that   $\mathcal{R}[p] \stackrel{\phi[p]}{\longrightarrow} \GGG[p] \toto M'$ is again a Lie groupoid extension.
    This construction still goes through under the weaker assumption that $p$ is a generalized surjective submersion for the Lie groupoid $\GGG\toto M $. Notice that $p$ is a generalized surjective submersion for the Lie groupoid $\GGG \toto M$
    if and only if it is a generalized surjective submersion for the Lie groupoid $\mathcal{R} \toto M$, so that we could say
    that this construction still goes through under the weaker assumption that $p$ be a generalized surjective submersion for the Lie groupoid ${\mathcal R}\toto M $.
     For all such maps $p:M'\to M$, we call the Lie groupoid extension $\mathcal{R}[p] \stackrel{\phi[p]}{\longrightarrow} \GGG[p] \toto M'$ the \emph{pull-back of the Lie groupoid extension $\mathcal{R} \stackrel{\phi}{\to} \GGG \toto M$ with respect to $p$}.

Having defined the pull-back of Lie groupoid extensions, we wish to define the pull-back of Lie groupoids $\G \to\HH $-extensions.
   This shall require to go through some technical considerations about the pull-back of the kernel and pull-back of the band of a Lie groupoid $\G$-extension.

   There is a clear notion of pull-back for both group bundles (resp. principal bundles): to say it in one word, given a group bundle (resp. principal bundle) $P \stackrel{\pi}{\to} M $, and a smooth map $p: M' \to M $, then the fibered product $P \times_{\pi,M,p} M' $ endows a natural structure of group bundle (resp. principal bundle).  To a Lie groupoid extension, we have associated in section \ref{subsec:defofextensions} a bundle of group, called the kernel, also starting with a $\G$-extension, we have constructed an $\Aut(\G) $-principal bundle, called the band. The next proposition claims that these two constructions behave well with respect to pull-back.

  \begin{prop}\label{prop:pullbackExtensions} Let $M,M'$ be smooth manifolds, $p : M' \to M $ be a surjective submersion map
  and $\mathcal{R}\stackrel{\phi}{\to} \GGG \toto M$ be a Lie groupoid extension.
   Then:\begin{enumerate}
   \item there is a canonical isomorphism between the kernel of the Lie grioupoid extension $\mathcal{R}[p] \stackrel{\phi[p]}{\longrightarrow} \GGG[p] \toto N$ and the pull-back of the kernel $K$ of the Lie groupoid extension $\mathcal{R}\stackrel{\phi}{\to} \GGG \toto M$ by the surjective submersion map $p$,
   \item if the Lie groupoid extension $\mathcal{R}\stackrel{\phi}{\to} \GGG \toto M$ is a Lie groupoid $\G$-extension, for some Lie group $\G$, then the pull-back
   $\mathcal{R}[p] \stackrel{\phi[p]}{\longrightarrow} \GGG[p] \toto N$ is also a Lie groupoid $\G$-extension,
   \item in the case of a $\G$-extension, there is a canonical isomorphism between the band of $\mathcal{R}[p] \stackrel{\phi[p]}{\lra} \GGG[p] \toto M'$ and the pull-back of the band of $\mathcal{R}\stackrel{\phi}{\to} \GGG \toto M$ by the surjective submersion$p$.
   \end{enumerate}
    The same holds true when the map $p$ is a generalized surjective submersion.
  \end{prop}
  \begin{proof}
  The kernel of the Lie groupoid extension $\mathcal{R}[p] \stackrel{\phi[p]}{\lra} \GGG[p] \toto M'$, denoted by $K[p]$, is, as a set, equal to $\{ (n,k,n) | n\in M' , k\in K_n \} $, where the kernel of the Lie groupoid extension $\mathcal{R} \stackrel{\phi}{\to} \GGG \toto M$ is denoted by $K$. As a bundle of group, $K[p]$ can be identified, therefore, with $M'\times_M K$. This proves the first item.

  In particular, the fiber $K[p]_{n}$ of the kernel $K[p]$ over a given point $n\in M'$ is isomorphic to
   $K_{p(n)} $, more generally, if $K$ is locally trivial with typical fiber $\G$, so is its pull-back $K[p]$. This  means precisely that the pull-back of a Lie groupoid $\G$-extension is again a Lie groupoid $\G$-extension.
   This proves the second item.

     The identification between $K' $ the pull back of the kernel of the Lie groupoid extension $\mathcal{R}\stackrel{\phi}{\to} \GGG \toto M$ and the kernel $K[p]$ of the Lie groupoid extension $\mathcal{R}[p] \stackrel{\phi[p]}{\to} \GGG[p] \toto M'$ induces an identification between
     the set of all Lie group automorphisms from $\G$ to $K_m'$ and $Band_{p(m)}(\mathcal{R} \stackrel{\phi}{\to} \GGG) $ for all $ m \in M'$. All together, these identifications yield an identification $Band(\mathcal{R}[p] \stackrel{\phi[p]}{\to} \GGG[p])$ and  $M'\times_{M} Band(\mathcal{R} \stackrel{\phi}{\to} \GGG)$. This proves the last item.
   \end{proof}

   We are now able to define clearly the notion of pull-back of a $\G \to \HH$-extension $( \mathcal{R} \to \GGG, P \to M, \chi ) $.
Let $p: M' \to M $ be a (maybe generalized) surjective submersion. According to the second item in proposition \ref{prop:pullbackExtensions}, the pull-back extension $ \mathcal{R}[p] \stackrel{\phi[p]}{\to} \GGG[p] $ is again a $\G$-extension. Moreover, $p^* P = P \times_M M' \to M' $ is an principal $\HH$-bundle over $M'$, which is acted upon by  $\mathcal{R}[p] \toto M' $ as follows:
    $$ (n,r,n') \bullet (x,n') = (r \bullet x , n) ,$$
    for all $n,n' \in M', x\in P, r \in R $ subject to the constraints $p(n)=s(r),t(r)=p(n')=p(x) $.
   The map $\chi[p] : P \times_M M'  \to Band(  \mathcal{R} \to \GGG ) \times_M M' $ defined by $ (p,n) \to (\chi(p), n) $, composed with the canonical isomorphism between $Band(  \mathcal{R} \to \GGG ) \times_M M' $ and
    $ Band(  \mathcal{R}[p]  \stackrel{\phi[p]}{\to}  \GGG[p] ) $ of item 3 in proposition \ref{prop:pullbackExtensions}, satisfies all the requirements needed to guarantee that   $( \mathcal{R}[p] \stackrel{\phi[p]}{\to} \GGG[p], p^*P \to M', \chi[p] ) $ is a $\G \to \HH$-extension.

    \begin{defi} Let $( \mathcal{R}  \stackrel{\phi}{\to} \GGG, P \to M, \chi ) $ be a Lie groupoid $\G \to \HH$-extension.
Let $p: M' \to M $ be a (generalized) surjective submersion. We call the Lie groupoid $\G \to \HH$-extension defined in the lines above the \emph{pull-back of the Lie groupoid $\G \to \HH$-extension $(\mathcal{R} \stackrel{\phi}\to \GGG, P \to M, \chi)$ with respect to $p$} and we denote it by
$  ( \mathcal{R}[p] \stackrel{\phi[p]}\to \GGG[p], P[p] \to M[p], \chi[p] ) $.
    \end{defi}

Indeed, we need a notion which is slightly more subtle. Recall that our purpose is to define gerbes as being the quotient of a sub-class of all $\G \to \HH$-extensions by some relation. We can now be more precise, and define, given a Lie groupoid $\BBB \toto \BBB_0 $,
a \emph{$\G \to \HH$-extension over $\BBB \toto \BBB_0$} to be a quadruple
$(q, \mathcal{R}\stackrel{\phi}{\to} \BBB[q],P \to M,\chi)$ where:
\begin{enumerate}
\item $q: M \to \BBB_0$ is a surjective submersion,
\item $ ( \mathcal{R} \stackrel{\phi}{\to} \BBB[q],P \to M,\chi)$ in a $\G \to \HH$-extension of the pull-back groupoid
$ \BBB[q] \toto M  $ of $\BBB \toto \BBB_0  $ with respect to $q$.
\end{enumerate}

We define the pull-back of those.

\begin{defi}
The pull-back of a Lie groupoid $\G \to \HH$-extension $(q, \mathcal{R}\stackrel{\phi}{\to} \BBB[q],P \to M,\chi)$  over the Lie groupoid $\BBB \toto \BBB_0$ w.r.t the surjective submersion $p: M'\to M$ is the Lie groupoid $\G \to \HH$-extension $(q\circ p,Y \stackrel{\phi[p]}{\to} \BBB[q \circ p],p^* P \to M',\chi[p])$ over the Lie groupoid $\BBB \toto \BBB_0$.
\end{defi}

\begin{rem}
 The previous definition used implicitly the existence of a natural isomorphism $ \BBB[q][p] \simeq  \BBB[q \circ p]$:
  $$ \xymatrix{  & &\BBB[q][M',p] \ar@<2pt>[d] \ar@<-2pt>[d]  \ar@{<->}^-{\simeq}[r] &  \BBB[M', q \circ p] \ar@<2pt>[dl] \ar@<-2pt>[dl]\\ & \BBB[q] \ar@<2pt>[d] \ar@<-2pt>[d] & M' \ar[dl]^p \ar@/^2pc/[ddll]^{q \circ p}\\  \BBB\ar@<2pt>[d] \ar@<-2pt>[d] & M\ar[dl]^q \\ \BBB_{0}  }  $$
  Indeed, the pull-back of the $\G \to \HH$-extension
  $( \mathcal{R}\stackrel{\phi}{\to} \BBB[q],P \to M,\chi)$ with respect to $p$ is a priori a $\G \to \HH$-extension
  of $\BBB[q][p]$. But in view of the
  isomorphism $ \BBB[q][p] \simeq  \BBB[q \circ p]$, it can be considered as a $\G \to \HH$-extension
   of $\BBB[q\circ p] $, and $(q \circ p, \mathcal{R}[p] \stackrel{\phi[p]}{\to} \BBB[q \circ p],p^* P \to M,\chi[p])$ is a $\G \to \HH$-extension over $\BBB \toto \BBB_0 $.
\end{rem}

\begin{rem}
As mentioned in Remark \ref{rem:linkGinotStienon}, and as pointed to us by the referee, a $\G\to \HH$ bundle over a Lie groupoid $\BBB$ in the sense of \cite{GinotStienon} will give in general a $\G \to \HH$-extension, but of a Lie groupoid $\GGG$ which is the pull-back of $\BBB$ through some surjective submersion onto the base manifold $\BBB_0$ of $\BBB $. This can be seen by using the fact that a morphism of $2$-groupoid
from a Lie groupoid $\GGG$ to the crossed module $\G \to \HH$ (seen as a Lie $2$-group) is in fact a simplicial map from the simplicial tower
of $\GGG$ to the simplicial tower of $\G \to \HH$. In turn, such a simplicial map is determined by a map $\varphi_1: \GGG \to \HH$
and a map from compatible pairs $\GGG_2$ to $\HH^2 \times \G$ of the form $(g_1,g_2) \to (\varphi(g_1),\varphi(g_2), \varphi_2(g_1,g_2))$
  The maps $\varphi_1$ and $ \varphi_2$ give a structure of Lie groupoid extension
by taking for $\HH$-principal bundle the set $P=M \times \HH$ (with $M$ the base manifold of $\GGG$) and ${\mathcal R} =  \GGG \times \G$.
The action of ${\mathcal R} =  \GGG \times \G$ on $P=M \times \HH $ is then given by:
 $$ (\gamma,g) \cdot (m,h):= (n, \rho(g) \varphi_1(\gamma) h) $$
 for all $g \in \G,h \in \HH , \gamma \in \GGG $ with $s(\gamma)=n, t(\gamma)=m$,
  while the product of  ${\mathcal R} =  \GGG \times \G$ is given by:
 $$ (\gamma,g) \cdot(\gamma',g') =(\gamma\gamma', gg' \varphi_2(\gamma,\gamma')) $$
 for all compatible $\gamma,\gamma'\in \GGG $ and all $ g,g' \in \G$.
As a consequence, a $\G\to \HH$-bundle over a Lie groupoid $\BBB \to \BBB_0$ in the sense of \cite{GinotStienon} defines a $\G \to \HH$-extension
 over $\BBB \to \BBB_0$, making the correspondence of remark \ref{rem:linkGinotStienon} more precise.
\end{rem}

We can now define the notion of Morita equivalence that we are interested in.

\begin{defi}
A Morita equivalence between two Lie groupoid $\G \to \HH$-extensions $(q,{\mathcal R} \stackrel{\phi}{\to} \BBB[q],P \to M,\chi)$ and
$(q',{\mathcal R} \stackrel{\phi}{\to} \BBB[q'],P \to M,\chi)$ over $\BBB \toto \BBB_0 $ is a triple
$(M'',p,p') $ where $M''$ is a manifold, $p: M'' \to M $ and $q:M'' \to M'$ are surjective submersions, such that:
\begin{enumerate}
\item the following diagram commutes:
 $$ \xymatrix{ & M'' \ar[dr]_{p} \ar[dl]^{p'} & \\ M' \ar[dr]^{q'} & & M \ar[dl]_{q} \\ & B & \\ }  $$

 $$ q \circ p=q' \circ p', $$
\item the pull-back of the Lie groupoid $ \G \to \HH$-extension $(\mathcal R \stackrel{\phi}{\to} \BBB[q],P \to M,\chi)$ with respect to $p$ is isomorphic
to the pull-back of the Lie groupoid  $ \G \to \HH$-extension $(\mathcal R' \stackrel{\phi}{\to} \BBB[q'],P' \to M',\chi')$ with respect to $p'$
(notice that both pull-back Lie groupoid $\G \to \HH $-extensions are $\G \to \HH $-extensions of $\BBB[q'\circ p'] = \BBB[q \circ p]$).
\end{enumerate}
\end{defi}
In terms of commutative diagram, Morita equivalence of Lie groupoid $\G \to \HH$-extensions over $\BBB \toto \BBB_0 $
can be visualized as follows
$$ \xymatrix{ & p^* P \ar@/^3pc/[rrrr]_{\simeq}  \ar[ddrr] & \mathcal R[p] \ar@/_1pc/[rr]^{\simeq} \ar[dr] & & \mathcal R[p'] \ar[dl]& p'^* P' \ar[ddll] & \\
         P \ar[ddr] & \mathcal R \ar[d]& & \BBB [q \circ p]\approx \BBB [q'\circ p'] \ar@<1pt>[d] \ar@<-1pt>[d]  & & \mathcal R' \ar[d] & P' \ar[ddl]\\
          & \mathcal B[q] \ar@<1pt>[d] \ar@<-1pt>[d] & & M'' \ar@{.>}[dll]^p \ar@{.>}[drr]_{p'}  & & \mathcal B[q'] \ar@<1pt>[d] \ar@<-1pt>[d] &\\
           & M \ar[drr]^{q} & & \mathcal B \ar@<1pt>[d] \ar@<-1pt>[d] & & M' \ar[dll]_{q'} & \\
            & & & \mathcal B_{0} & & & }$$
          \begin{examp} \label{ex:isom}
A pair  $(q,\mathcal R \stackrel{\phi}{\to} \BBB[q],P \to M,\chi)$
and  $(q,\mathcal R' \stackrel{\phi}{\to} \BBB[q],P'\to M,\chi)$ of
$\G \to \HH$-extensions over $\BBB \toto \BBB_0 $ which are isomorphic over the identity of
$\BBB[q]$ are Morita equivalent.
\end{examp}
\begin{examp} \label{ex:pullback}

Every Lie groupoid $\G \to \HH$-extension over a Lie groupoid is Morita equivalent to its pull back with respect to a (generalized) surjective submersion.
\end{examp}

We can not say, strictly speaking, that Morita equivalence of $\G \to \HH$-extensions over a given Lie groupoid
$\BBB \toto \BBB_0 $ is an equivalence relation because $\G \to \HH$-extensions over $\BBB \toto \BBB_0 $ do not form a set. However, the axioms of equivalence relations remain satisfied, as shown in the next proposition

\begin{prop}
Let $\BBB \toto \BBB_0 $ be a Lie groupoid.
\begin{enumerate}
\item A $\G \to \HH$-extension over $\BBB \toto \BBB_0 $ is always Morita equivalent to itself.
\item Let $Ext_1,Ext_2$ be $\G \to \HH$-extensions over $\BBB \toto \BBB_0 $. $Ext_1$ is Morita equivalent to $Ext_2$
if and only if $ Ext_2$ is Morita equivalent to $Ext_1$.
\item Let $Ext_1,Ext_2,Ext_3$ be $\G \to \HH$-extensions over $\BBB \toto \BBB_0 $. If
$ Ext_1$ is Morita equivalent to $Ext_2$ and $ Ext_2$ is Morita equivalent to $Ext_3$, then $ Ext_1$ is Morita equivalent to $Ext_3$.
\end{enumerate}
\end{prop}
\begin{proof}
Only the third item merits some justification. If $M$, together with the surjective submersions $p,q$ give a
Morita equivalence between $Ext_1 $ and $Ext_2 $ while $M' $ together with the surjective submersions $p',q'$ give a
Morita equivalence between $Ext_2 $ and $Ext_3 $, then we introduce $ M'':= M \times_{q,M_2,p'} M' $
and equip it with the surjective submersions $(m,m') \to p(m) $ and $(m,m') \to q'(m') $ onto $M_1$
and $M_2 $ respectively, where, in the previous, $M_i, i=1,2,3 $ is the base manifold of the $\G \to \HH$-extension $Ext_i $.
A cumbersome but easy computation shows that the pull-back of $Ext_1 $ and $Ext_3 $ to $ M'' $ are isomorphic Lie groupoid $\G \to \HH$-extensions.
\end{proof}

This proposition allows one to give, at last, the following definition.

\begin{defi}
A $\G \to \HH$-gerbe over $ \BBB \toto \BBB_0$ is a Morita equivalence class of Lie groupoid $\G \to \HH$-extensions over $ \BBB \toto \BBB_0$.
\end{defi}

To justify this definition, we shall in subsection \ref{subsec:manifoldcase} show that, when the Lie groupoid $ \BBB \toto \BBB_0$ is simply a manifold
 Lie groupoid $ \BBB \toto \BBB_0$, $\G \to \HH$-gerbe are precisely the same thing as  $\G \to \HH$-valued non-Abelian $1$-cohomology.

\subsection{The manifold case: $\G \to \HH$-gerbes as non-Abelian $1$-cohomology}\label{subsec:manifoldcase}

The notion of  $\G \to \HH$ non-Abelian $1$-cohomology w.r.t. a given open covering was introduced in section \ref{subsec:mfdcase} .
As usual, $\G \to \HH$ non-Abelian $1$-cohomology is obtained by inductive limits of those.
More precisely, we proceed as follows. By a \emph{refinement} of an open cover ${\mathcal U}=(U_i)_{i \in I} $, we mean a pair $ ({\mathcal V},\sigma)$ made of an open cover $ {\mathcal V}=(V_j)_{j \in J}$ together with a map $\sigma: J \to I $ such that $ V_j \subset U_{\sigma (j)} $ for all $j \in J$.
Notice that $\sigma $ induces a map, again denoted by $\sigma $, from $\coprod_{k,l \in J}V_{kl} $ to
$\coprod_{i,j\in I}U_{ij} $ (resp. $\coprod_{k,l,m \in J}V_{klm} $ to
$\coprod_{i,j,k\in I}U_{ijk} $), obtained by mapping $x_{kl} \in V_{kl}  $
to $x_{\sigma(k)\sigma(l)} \in U_{\sigma(k)\sigma(l)} $ (using the notations of section \ref{notation:Cech}).
By the \emph{pull-back} of a non-Abelian $1$-cocycle $(\lambda,\gggg)   \in \mathcal C ^{\infty}( \coprod_{i,j \in I}V_{ij},\HH ) \times \mathcal C ^{\infty} (
\coprod_{i,j,k \in J}V_{ijk},\G)  $ w.r.t. ${\mathcal U}$, we mean
the pair of functions $(\sigma^* \lambda,\sigma^* \gggg )$ in $\mathcal C ^{\infty}( \coprod_{i,j \in I}V_{ij},\HH ) \times \mathcal C ^{\infty} (
\coprod_{i,j,k \in J}V_{ijk},\G)  $. Notice that, by construction, $ (\sigma^* \lambda)_{ij} = \left. \lambda_{\sigma(i)\sigma(j)} \right|_{V_{ij}} $
and
 $  (\sigma^* \gggg)_{ijk} =  \left.\gggg_{\sigma(i)\sigma(j)\sigma(k)} \right|_{V_{ijk}} $
 for all $i,j,k \in J $.

\begin{lem}
Let  $({\mathcal V},\sigma)$  be a refinement of ${\mathcal U} $.
The pull-back of a $\G \to \HH $-valued non-Abelian $1$-cocycle w.r.t. ${\mathcal U} $  is   a $\G \to \HH $-valued non-Abelian $1$-cocycle w.r.t ${\mathcal V} $. Moreover, two $\G \to \HH $-valued non-Abelian $1$-cocycles that differ by a coboundary have pull-back that differ by a coboundary again.
\end{lem}

We now identify two $\G \to \HH $-valued non-Abelian $1$-cocycles $(\lambda, \gggg  ) $ and $(\lambda', \gggg') $, defined on covering $ \mathcal U  $ and $\mathcal U' $ of $N$ respectively, if there exists a common refinement of both
$ \mathcal U  $ and $\mathcal U' $  such that the pull-back to that refinement of
 $(\lambda, \gggg  ) $ and $(\lambda', \gggg') $ differ by a coboundary. We denote by $H^1(\G \to \HH ) $ the set henceforth obtained and we call this set the \emph{$\G \to \HH $-valued non-Abelian $1$-cohomology on $N$}. A priori, $H^1(\G \to \HH ) $ has no group structure.

We can now state the main result of this section.

\begin{them}\label{th:coho=gerbes}
Let $N$ be a manifold. There is a one-to-one correspondence between:
\begin{enumerate}
\item $\G \to \HH$-valued non-Abelian $1$-cohomology on $N$,
\item $\G \to \HH$ gerbes over $N \toto N $.
\end{enumerate}
\end{them}


The proof of the theorem requires two lemmas. For the first one, recall from proposition \ref{prop:cocycles}
that, given an open covering  $\mathcal U $ of $N$, there is a one-to-one correspondence between
non-Abelian $1$-cocycles and adapted extensions of the \v{C}ech groupoid $N[\mathcal U]$.

\begin{lem}\label{lem:lemma A}
Let  $({\mathcal V},\sigma) $ be a refinement of ${\mathcal U}$ and $(\lambda,\gggg)$ be a non-Abelian $1$-cocycle w.r.t.~${\mathcal U}$. Then the adapted Lie groupoid $\G \to \HH$-extension associated to the pull-back of the non-Abelian $1$-cocycle $(\lambda, \gggg) $ is isomorphic (as a $\G \to \HH$-extension) to the pull-back of the adapted Lie groupoid $\G \to \HH$-extension associated to $(\lambda, \gggg) $. This can be expressed as the commutativity (up to a a canonical isomorphism) of the diagram:

$$ \xymatrix{*\txt{non.Ab. $1$-cocycle \\ w.r.t ${\mathcal V}$ } \ar@{<->}[rrr]^{\txt {corresp. of \\prop. \ref{prop:cocycles}}} &&& *\txt{adapted ext. of \\ $N[ {\mathcal V}]$ }   \\
&&& \\
             *\txt{ non.Ab. $1$-cocycle \\ w.r.t ${\mathcal U}$ }\ar[uu]^{pull-back} \ar@{<->}[rrr]^{\txt {corresp. of \\prop. \ref{prop:cocycles}}} &&& *\txt{adapted ext. of \\ $N[ {\mathcal U}]$ } \ar[uu]_{pull-back}}$$

\end{lem}
\begin{proof}
Let  $(\mathcal U , \bullet , \star)$  (resp. $(\mathcal V , \bullet' , \star')$) be the adapted Lie groupoid $\G \to \HH$-extension associated to the $\G \to \HH$-valued non-Abelian $1$-cocycle $(\lambda,\gggg)$ (resp. $(\lambda',\gggg')$, the pull-back of $(\lambda,\gggg)$ w.r.t. $\sigma$). Let $\sigma $ stand for the map $ \coprod_{j \in J} V_j\to \coprod_{i\in I} U_i$ (resp. $ \coprod_{k,l \in J} V_{kl}\to \coprod_{i,j\in I} U_{ij}$). The pull-back Lie groupoid  $ (\G \times \coprod_{i,j \in I}U_{ij})[\sigma] $
is isomorphic to $\G \times \coprod_{i,j \in  J}V_{ij}$ through the isomorphism defined for all $i,j \in J, x \in V_{ij}, g \in \G$ by
 $$ \psi (x_{i} ,(g,x_{\sigma(i)\sigma(j)}),x_{j}):=(g,x_{ij}),$$
 and the map  $$\psi' (x_{i} ,(x_{\sigma(i)},h)):=(x_i, h)$$ is an isomorphism between the pull-back of $\coprod_{i \in I} U_i \times \HH $  through $\sigma $
  and $ \coprod_{j\in J} V_j \times \HH $. We leave it to reader to prove that $(\psi,\psi',id_{\HH})$ is an isomorphism of Lie groupoid $\G \to \HH$-extensions.
\end{proof}

The next lemma shall also have its importance. The reader can replace the Lie groupoid $\BBB \toto \BBB_0$ by $N \toto N $ for the sake of simplicity, since we shall only use the lemma in that case.


\begin{lem} \label{lem:lemma B}
Let $(q,{\mathcal R} \stackrel{\phi}{\to} \BBB[q],P \to M,\chi)$ be a Lie groupoid $\G \to \HH$-extension over $\BBB \toto \BBB_0$. Let $\tau: M' \to M $ be a map such that $ q \circ \tau $ is a surjective submersion.
Then:
\begin{enumerate}
\item $\tau $ is a generalized surjective submersion for both Lie groupoids $ \mathcal{R}\toto M $ and $\BBB[q] \toto M $.
\item $( q \circ \tau, \mathcal{R}[\tau] \stackrel{\phi[\tau]}{\to}\BBB[q \circ \tau ],\tau^* P \to M', \chi[\tau])$
is a Lie groupoid $\G \to \HH$-extension.
\item This Lie groupoid $\G \to \HH$-extension is Morita equivalent (over the identity of $\BBB \toto \BBB_0$) to $(q,\mathcal{R} \stackrel{\phi}{\to} \BBB[q],P \to M,\chi)$.
  \end{enumerate}
\end{lem}
\begin{proof}
We wish to show that the map $\xi: M' \times_{\tau,M,s} B[q]\to M $ given by $(b,m)\mapsto t(b)$ for all $b\in B[q], m\in M$ is a surjective submersion. Let $m\in M$, take $m'\in (q\circ \tau)^{-1}(q(m))$ (which is non-empty by assumption). Now since $t^{-1}(q(m))$ is not empty so $(m',(\tau(m'),b,m))$
projects on $m$ by $\xi $, where $b\in t^{-1}(q(m)) $. This proves the surjectivity. To check that $\xi $ is indeed a submersion, we have to think in terms of infinitesimal paths.
Let $(m' , (\tau(m'),b,m ))  $ be a point in $M' \times_{\tau,M,s} B[q]$, and $m \in M$ such that $ \xi(m' , (\tau(m'),b,m ))=m$. Let
 $m (\epsilon) $ be a path in $M$ starting from $m $. Since the target map (of Lie groupoid $\BBB \toto \BBB_0 $) is always a surjective submersion there exists a path $b(\epsilon)$ in $\BBB $ starting at $b$ such that $t(b(\epsilon))=q(m(\epsilon)) $ (for all $\epsilon $ small enough). Since $q \circ \tau $ is a surjective submersion by assumption, there exists also a path $m'(\epsilon) $ in $M'$ starting at $m'$
such that $ q \circ \tau (m'(\epsilon)) =s(b(\epsilon)) $ for $\epsilon  $ small enough. By construction, the path $ (m'(\epsilon),\tau(m'(\epsilon), b(\epsilon), m(\epsilon) ))  $ is a path in  $ M' \times_{\tau,M,s} B[q]\to M$ which starts at $(m' , (\tau(m'),b,m ))  $ and projects by $\xi$ onto $m(\epsilon) $, which completes the proof of the first item.

In view of the proof of lemma \ref{lem: generalized surjective submersion}, all the algebraic axioms of Lie groupoid $\G \to \HH$-extensions are satisfied
by $(\mathcal R[\tau] \stackrel{\phi}{\to}\BBB[q \circ \tau ],\tau^* P \to M',\tau^* \chi)$. Lemma \ref{lem:manifold} implies that the sets involved are manifolds. This completes the proof of the second item.

For the last item, the manifold that we shall consider to construct an explicit Morita equivalence is:
   $$  T= M' \times_{ \tau, M, t } \mathcal R $$
   equipped with the surjective submersions  $q_M' :T \to M'$ and $q_M :T \to M$ given by the projection on the first component and the target of the second component respectively.
    By construction, the following diagram commutes:
    $$ \xymatrix{ & T \ar[dr]_{q_{M'}} \ar[dl]^{q_{M}}& \\ M \ar[dr]^{q} & & M' \ar[ll]_{\sigma} \ar[dl]_{q \circ \tau} \\ & B& \\ } $$
    This implies that $ \mathcal R[\tau][q_{M'}] \simeq \mathcal R[ \tau \circ q_{M'}] = \mathcal R[q_M] $ and also $$ p_{M'}^* \sigma^* P \simeq  (\sigma \circ q_{M'})^* P= p_M^* P.$$
     It is routine to check that this pair of isomorphisms form an isomorphism of Lie groupoid $\G \to \HH$-extensions between the pull back of $( q \circ \tau, \mathcal R[\sigma] \stackrel{\phi}{\to}\BBB[q \circ \tau ],\sigma^* P \to M',\tau^* \chi , )$ w.r.t. $q_{M'}$ and the pull back of $( q , \mathcal R \stackrel{\phi}{\to}\BBB[q ],P \to M',\chi , )$ w.r.t. $q_M$.

\end{proof}
We now prove theorem  \ref{th:coho=gerbes}.
\begin{proof}
According to the first item of proposition \ref{prop:cocycles},
 to an arbitrary $\G \to \HH$-valued non-Abelian $1$-cocycle $( \lambda,\gggg) $ with respect to an arbitrary open cover ${\mathcal U}$ corresponds an
 adapted Lie groupoid $\G \to \HH$-extension (which is by construction a Lie groupoid $\G \to \HH$-extension above the Lie groupoid $N \toto N $).

This assignment goes to the quotient to yield an assignment from  $\G \to \HH$-valued non-Abelian $1$-cohomology on $N$ to $\G \to \HH$-gerbes over $N \toto N $. This follows from the fact that the adapted Lie groupoid $\G \to \HH$-extensions associated to a
$\G \to \HH$-valued non-Abelian $1$-cocycle and a pull-back of it are Morita equivalent over the identity of $N \toto N $ by Lemma \ref{lem:lemma A}.
Also, by proposition \ref{prop:coboundaries},  the adapted extensions associated to two $\G \to \HH$-valued
non-Abelian $1$-cocycles that differ by a coboundary are isomorphic, hence Morita equivalent over the identity of $N \toto N$
by example \ref{ex:isom}.
Hence, the adapted Lie groupoid $\G \to \HH$-extensions associated to two
$\G \to \HH$-valued non-Abelian $1$-cocycles that define the same element in cohomology are  Morita equivalent over the identity of $N \toto N$, yielding a well-defined map from  $\G \to \HH$-valued non-Abelian $1$-cohomology on $N$ to $\G \to \HH$-gerbes over $N \toto N$, that we denote by $\Xi $.

 We first check that $\Xi $ is surjective. Let $(q,\mathcal R \stackrel{\phi}{\to} N[q],P \to M,\chi)$ be an arbitrary Lie groupoid $\G \to \HH$-extension over $N \toto N$. There exists an open cover $\mathcal U=(U_i)_{i \in I}$ of $N$ such that $q: M \to N $
 admits local sections $\sigma_i : U_i \to M $ for all $i \in I $, which, altogether, define a map $\sigma: \coprod_{i \in I} U_i \to M $.
 By lemma \ref{lem:lemma B}, $(q,\mathcal R \stackrel{\phi}{\to} N[q],P \to M,\chi)$  is Morita equivalent over the identity of $N \toto N$ to
 its pull-back with respect to $\sigma $. The pull-back being a Lie groupoid $\G \to \HH$-extension of the \v{C}ech groupoid is, by proposition \ref{prop:everyextensionisadapted}, isomorphic (hence Morita equivalent by example \ref{ex:isom}) to an adapted one. Hence $(q,\mathcal R \stackrel{\phi}{\to} N[q],P \to M,\chi)$ is Morita equivalent to an adapted Lie groupoid $\G \to \HH$-extension, which, by proposition  \ref{prop:cocycles}, comes from some non-Abelian $1$-cocycle. This proves that the assignment $\Xi $ is surjective.

 We then check that $ \Xi$ in injective. The proof is based on  the following general property of open covers.
 Assume that there is a commutative diagram as follows:
  $$  \xymatrix{  & M' \ar[dr]^p \ar[dl]_{p'} & \\ \coprod_{i \in I} U_i \ar[dr]_\imath & & \coprod_{j\in J} V_j \ar[dl]^\imath \\ & N &   }$$
  with $p$ and $p'$ surjective submersions (above, the symbol $ \imath$ stands for all the canonical inclusions, and  $(U_i)_{i \in I} $ and $(V_j)_{j \in J} $ are open covers of the manifold $N$). Then there is a common refinement $(W_k)_{k \in K} $ of $(U_i)_{i \in I} $ and $(V_j)_{j \in J} $ and a map $\tau: \coprod_{k \in K} W_k  \to M'$ such that the following diagram commutes:
   \begin{equation}\label{eq:commonrefinement}\xymatrix{ & \coprod_{k \in K} W_k \ar[d]^\tau \ar[ddr]^\imath \ar[ddl]_\imath & \\ & M' \ar[dr]_p \ar[dl]^{p'} & \\ \coprod_{i \in I} U_i \ar[dr]^\imath & & \coprod_{j\in J} V_j \ar[dl]_\imath \\ & N &   }
   \end{equation}
  where, again, we use the symbol $\imath $ to denote all the canonical inclusions.

Assume now two $\G \to \HH$-valued non-Abelian $1$-cocycles, defined w.r.t. open covers  $(U_i)_{i \in I} $ and $(V_j)_{j \in J} $ respectively, have adapted Lie groupoid $\G \to \HH$-extensions $Ext1 $  and $Ext2 $ (which are hence over the Lie groupoid $N\toto N$) associated with which are Morita equivalent.
By the very definition of Morita equivalence of $\G \to \HH$-extensions, this implies that there exists a manifold $M'$ together with surjective submersions $p :M' \to \coprod_{i\in I}U_i$ and $p' :M' \to \coprod_{j\in J}V_j$  such that $\iota \circ p = \iota \circ p' $ and such that the pull-backs $Ext1[p] $  and $Ext2[p'] $ of both extensions
to $M'$ are isomorphic. According to the discussion above, there exists a common refinement
 $\coprod_{k \in K} W_k$ of both open covers and a map  $\tau: \coprod_{k \in K} W_k \to M' $
   such that the diagram (\ref{eq:commonrefinement}) commutes. According to lemma \ref{lem:lemma B}, the pull-back of the adapted extension $Ext1 $
   on $\coprod_{k \in K} W_k$ is isomorphic to the pull-back of $Ext1[p] $ by $\sigma $. Similarly,  the pull-back of the adapted extension $Ext2 $ on $\coprod_{k \in K} W_k$ is isomorphic to the pull-back of $Ext2[p] $ by $\sigma $. Since $Ext1[p] $ and $Ext2[p] $ are isomoprhic, this implies that the pull-back of $Ext1 $
   and $Ext2 $ to $\coprod_{k \in K} W_k$ are isomorphic.
   According to Lemma \ref{lem:lemma A}, this means that the pull-back of both cocycles to $(W_k)_{k \in K} $ have corresponding adapted extensions that are isomorphic. By proposition \ref{prop:coboundaries}, it means that their pull-back to  $(W_k)_{k \in K} $ differ by a coboundary, i.e. that both cocycles define the same class in cohomology. This proves the injectivity.
\end{proof}

\subsection{$\G \to \HH$-gerbes over differentiable stacks.}

Recall that a Morita equivalence between two Lie groupoids $ \BBB \toto \BBB_0$
and $\BBB' \toto \BBB_0' $ is a quadruple ${\mathcal M}=(T,f,g,\Phi) $, with $T$ a manifold, $f,g$ surjective submersions from $T$ to $\BBB_0$
 and to $\BBB_0' $ respectively, and $\Phi$ a Lie groupoid isomorphism over the identity of $T$ between
 $\BBB[f] \toto T$ and $\BBB'[g] \toto T $.
 (Alternatively, Morita equivalence may be defined with the help of the notion of bi-modules, a description which happens to be
  equivalent to the previous one, see \cite{BehrendXu}.)
  Morita equivalent Lie groupoids often share similar properties, in particular which regards to cohomology.
  The next theorem shows that they also have the same gerbes over them.

\begin{them}\label{thm:morita_and_gerbes}
A Morita equivalence between two groupoids $ \BBB \toto \BBB_0$
and $\BBB' \toto \BBB_0' $ induces a one-to-one correspondence between:
\begin{enumerate}
\item $\G \to \HH$-gerbes over $ \BBB \toto \BBB_0 $,
\item $\G \to \HH$-gerbes over $ \BBB' \toto \BBB_0' $.
\end{enumerate}
\end{them}
\begin{proof}
Let ${\mathcal M}=(T,f,g,\Phi)$ be a Morita equivalence between the Lie groupoids $ \BBB \toto \BBB_0$ and $\BBB' \toto \BBB_0' $, i.e. $f:T \to \BBB_0$ and $g:T \to \BBB'_0$ are surjective submersions and $\Phi: \BBB[f] \to \BBB'[g]$ is an isomorphism of Lie groupoids between the the pull-back groupoids $ \BBB[f] \toto T$ and $\BBB'[g] \toto T $.
We intend to assign to an arbitrary Lie groupoid $\G \to \HH$-extension $Ext:=(q, X\stackrel{\phi}{\to} \BBB[q], P\to M, \chi)$ over $ \BBB \toto \BBB_0 $ a  Lie groupoid $\G \to \HH$-extension over $\BBB' \toto \BBB_0'$.
 To start with, we consider the set $M':= M\times_{q,\BBB_{0},f}T$. One checks easily that $M'$ is a manifold such that the projections $\alpha,\beta $ onto the first and second components are surjective submersions,
 and as well as the maps $q \circ \alpha : M' \to \BBB_0$ and $q':= g \circ \beta : M'\to \BBB_0' $.
Then we consider the pull-back of $Ext $ by $\alpha $, namely
 $$ ( \mathcal R[\alpha]\stackrel{\phi[\alpha]}{\to} \BBB[q \circ \alpha], \alpha^* P\to  M', \chi[\alpha]) .$$
We claim that there exists an isomorphism $\Phi': \BBB[q \circ \alpha] \to \BBB'[g \circ \beta]$, so that
$$ (g \circ \beta ,  \mathcal R[\alpha]\stackrel{\Phi ' \circ \phi[\alpha]}{\longrightarrow} \BBB'[g \circ \beta], \alpha^* P\to  M', \chi[\alpha]) $$
  is a Lie groupoid $\G \to \HH$-extension over $\BBB' \toto \BBB_0' $.
For all $((m_1,t_1),b,(m_2,t_2)) \in M' \times_{q \circ \alpha , \BBB_0 , s} B  \times_{t, \BBB_0,q \circ \alpha } M'$, we have the relations
  $$ f(t_1)=q(m_1), f(t_2)=q(m_2) , s(b)=q \circ \alpha (m_1,t_1)=q(m_1) , t(b)=q \circ \alpha (m_2,t_2)=q(m_2) $$
which impliy that $f(t_1)=s(b), f(t_2)=t(b)$, hence that $(t_1,b,t_2) \in \BBB[f]$. This allows us to set
  $$ \Phi'((m_1,t_1),b,(m_2,t_2)):=((m_1,t_1),b',(m_2,t_2))$$
where $b' \in \BBB$ is given by $\Phi(t_1,b,t_2)=(t_1,b',t_2)$.
By construction,  $((m_1,t_1),b',(m_2,t_2)) $ is in $B'[ g \circ \beta]$,
and it is routine to check that $\Phi'$ is an isomorphism of Lie groupoids.

 %
 We have therefore assigned  a Lie groupoid $\G \to \HH$-extension over
 $\BBB' \toto \BBB_0' $ to a Lie groupoid  $\G \to \HH$-extension over
 $\BBB \toto \BBB_0 $, as intended. In picture:
            $$
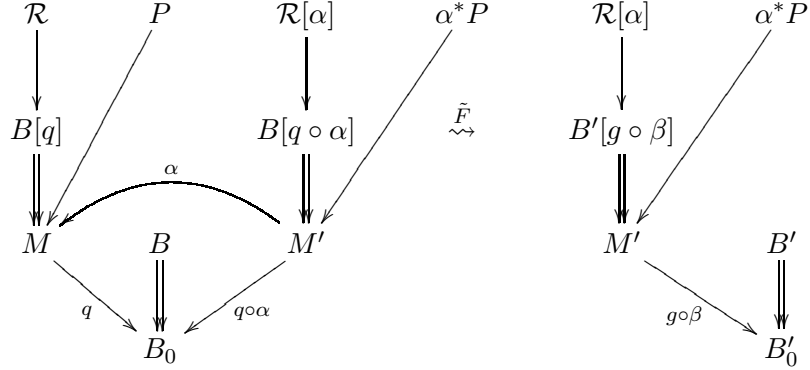
\xymatrix{\mathcal R \ar[d] & P \ar[ddl] & \mathcal R[\alpha] \ar[d] & \alpha ^* P \ar[ddl] & \mathcal R[\alpha] \ar[d] & \alpha ^* P \ar[ddl] \\
                      B[q]\ar@<1pt>[d] \ar@<-1pt>[d]& & B[q \circ \alpha] \ar@<1pt>[d] \ar@<-1pt>[d] & \stackrel{\tilde{F}}{\rightsquigarrow} & B'[g \circ \beta]\ar@<1pt>[d] \ar@<-1pt>[d]\\
                      M \ar[dr]_{q} & B \ar@<1pt>[d] \ar@<-1pt>[d] & M' \ar@/_2pc/[ll]_{\alpha}\ar[dl]^{q \circ \alpha} & & M' \ar[dr]_{g \circ \beta}& B' \ar@<1pt>[d] \ar@<-1pt>[d]\\
                         & B_{0} & & & & B'_{0}  } $$
                         $$\alpha(m,n)=m, q'(m,n)=g(n)$$
 The same construction could be done to assign a $\G \to \HH$-extension over
 $\BBB' \toto \BBB'_0 $ to a $\G \to \HH$-extension over
 $\BBB \toto \BBB_0 $.
Since the roles of $\BBB \toto \BBB_0 $ and
 $\BBB' \toto \BBB_0' $ can be exchanged, in order to check that both assignments induce one-to-one correspondence between the corresponding gerbes, it is necessary and sufficient to check
 that:
\begin{itemize}
\item[(i)] Morita equivalent Lie groupoid $\G \to \HH$-extensions over the Lie groupoid $\BBB \toto \BBB_0 $ are mapped to Morita equivalent Lie groupoid $\G \to \HH$-extensions over the Lie groupoid $\BBB' \toto \BBB'_0 $ by the first assignment.
\item[(ii)] applying the first assignment, then the second one to a Lie groupoid $\G \to \HH$-extension $Ext $ over $\BBB \toto \BBB_0$ yields a $\G \to \HH$-extension which is Morita equivalent to $Ext $.
\end{itemize}
Let us check these two points. The second one is an immediate consequence of example \ref{ex:pullback}, since a Lie groupoid $\G \to \HH$-extension over $\BBB \toto \BBB_0 $ is always Morita equivalent to its pull-back.
The first one is more involved. Let $Ext_1,Ext_2$ be two Lie groupoid $\G \to \HH$-extensions over $\BBB \toto \BBB_0 $, namely, to fix notations
 $$  Ext_i:=( q_i, \mathcal R_i\stackrel{\phi_i}{\to} \BBB[q_i], P_i\to M_i, \chi_i )  \hbox{ , } i=1,2, $$
 and let
 $$Ext_i':= (g \circ \beta ,  \mathcal R_i[\alpha]\stackrel{\Phi ' \circ \phi_i[\alpha]}{\longrightarrow} \BBB'[g \circ \beta], \alpha^* P\to  M'_i, \chi_i[\alpha]) \hbox{ , } i=1,2. $$
be the associated Lie groupoid $\G \to \HH$-extensions over $\BBB' \toto \BBB'_0 $ constructed as above.
Assume now that $Ext_i \hbox{ , } i=1,2.  $ are Morita equivalent. This means, first, that there is a commutative diagram of surjective submersions:
  $$
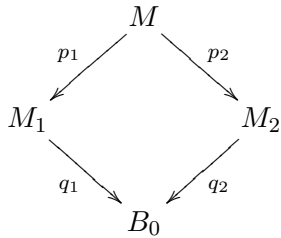
\xymatrix{ & M \ar[dr]^{p_2} \ar[dl]_{p_1}& \\ M_1 \ar[dr]_{q_1} & & M_2 \ar[dl]^{q_2} \\ & \BBB_0& \\ }   $$
 This implies that the following is also a commutative diagram of surjective submersions:
  $$\xymatrix{ &  M\times_{\BBB_{0}}T\ar[dr]^{(p_2,id_T)} \ar[dl]_{(p_1,id_T)}& \\ M'_1 \ar[dr]_{\beta \circ g} & & M'_2 \ar[dl]^{\beta \circ g} \\ & \BBB_0& \\ }   $$
  where $M'_i:=M_i \times_{\BBB_{0}}T$ and the fibred product $M\times_{\BBB_{0}}T$ are considered w.r.t. the maps $q_1 \circ p_1 (=q_2 \circ p_2): M \to \BBB_0$ and $f: T \to \BBB_0$.

  To show that both pull-back of $Ext'_i$ w.r.t. $(p_i,id_T)$ with $i=1,2$ are isomorphic as $\G \to \HH $-extensions, it suffices to show that $\mathcal R_i[\alpha \circ (p_i,id_T)] \toto M'_i$ with $i=1,2$ are isomorphic Lie groupoids. But this is a consequence of the existence of isomorphism between $\mathcal R_i[p_i] \toto M_i \hbox{ , } i=1,2$ which is in turn a consequence of the assumption of $(M,p_1,p_2,\phi)$ being a Morita equivalence between $Ext_1, Ext_2$.

\end{proof}

Let us denote by $F({\mathcal M}) $ the correspondence between $\G \to \HH$-gerbes associated to a given Morita equivalence ${\mathcal M} $.
It is easy to check that, given ${\mathcal M}_1$ and ${\mathcal M}_2$ composable Morita equivalences, the relation  $F({\mathcal M}_1 ) \circ F( {\mathcal M}_2)
 = F({\mathcal M}_1 \circ {\mathcal M}_2) $ holds when the composition ${\mathcal M}_1 \circ {\mathcal M}_2$ of Morita equivalences,
 is defined as in \cite{BehrendXu}.
According to \cite{BehrendXu}, Lie groupoids up to Morita equivalence are one possible description  of differential stacks,
so that theorem \ref{thm:morita_and_gerbes} makes sense of the notion of $\G \to \HH$-gerbes valued in a differential stack.

\section{Acknowledgments} I would like to thank the University of Metz for allowing me to stay there in autumn 2012.
This work was
supported by CMUC-FCT (Portugal) and FCT grants SFRH/BD/33810/2009, PTDC/MAT/099880/2008 and PEst-C/MAT/UI0324/2011
 through European program COMPETE/FEDER.
 Last, I would like to express my gratitude to the anonymous referee for his enlightening comments.

\bibliographystyle{plain}	 
\bibliography{myrefrences}		 
\end{document}